\documentclass{amsart}
\usepackage{amssymb}

\numberwithin{equation}{section}

\theoremstyle{plain}
\newtheorem{theorem}{Theorem}[section]
\newtheorem*{theorem*}{Theorem}

\newtheorem{proposition}[theorem]{Proposition}
\newtheorem*{proposition*}{Proposition}
\newtheorem{lemma}[theorem]{Lemma}
\newtheorem*{lemma*}{Lemma}
\newtheorem{corollary}[theorem]{Corollary}
\newtheorem*{corollary*}{Corollary}

\newtheorem*{theorem:dPV}{Theorem (dPV)}
\newtheorem*{theorem:dPVI}{Theorem (dPVI)}
\newtheorem*{proposition:isomonodromy}{Proposition (isomonodromy transformation)}

\theoremstyle{definition}
\newtheorem{definition}[theorem]{Definition}
\newtheorem*{definition*}{Definition}

\newtheorem*{notation*}{Notation}

\theoremstyle{remark}
\newtheorem{remark}[theorem]{Remark}
\newtheorem*{remark*}{Remark}

\newtheorem*{remarks*}{Remarks}
\newtheorem{example}[theorem]{Example}
\newtheorem*{example*}{Example}

\newcommand\eqd{:=}                     
\newcommand\p[1]{{\mathbb{P}^{\mathbf{#1}}}}        
\newcommand\iso{{\ \widetilde\to\ }}                
\newcommand\isorat{{\stackrel\sim\dashrightarrow}}     
\newcommand\mtop[1]{\mathop{\mathrm{#1}}\nolimits}
\newcommand{\sing}{\mtop{Sing}}             
\newcommand{\tr}{\mtop{tr}}                 
\newcommand{\im}{\mtop{im}}                 
\newcommand{\res}{\mtop{res}}                   
\newcommand{\diag}{\mtop{diag}}             
\newcommand{\detrg}{\mtop{detR\Gamma}}

\newcommand\dotminus{\mathop{\dot{-}}}

\newcommand\C{\mathbb{C}}
\newcommand\Z{\mathbb{Z}}

\newcommand\cF{{\mathfrak F}}
\newcommand\cG{{\mathfrak G}}
\newcommand\cX{{\mathfrak X}}
\newcommand\cY{{\mathfrak Y}}
\newcommand\tA{{\widetilde A}}

\newcommand\vL{{\mathcal L}}
\newcommand\vA{{\mathcal A}}

\newcommand\vR{{\mathcal R}}
\newcommand\vO{{\mathcal O}}

\author{D.~Arinkin, A.~Borodin}
\title[Tau-function of discrete isomonodromy transformations]{Tau-function of discrete isomonodromy transformations
 and probability}
\date{April 17, 2007}

\begin{document}
\maketitle
\begin{abstract}
We introduce the $\tau$-function of a rational d-connection and its
isomonodromy transformations. We show that in a continuous limit our
$\tau$-function agrees with the Jimbo-Miwa-Ueno $\tau$-function,
compute the $\tau$-function for the isomonodromy transformations
leading to difference Painlev\'e V and difference Painlev\'e VI
equations, and prove that the gap probability for a wide class of
discrete random matrix type models can be viewed as the
$\tau$-function for an associated d-connection.

\end{abstract}

\section*{Introduction}

The theory of isomonodromy deformations of rational connections over
$\p1$ has a long history. It was pioneered in the beginning of the
twentieth century by R.~Fuchs and L.~Schlesinger, and after being
dormant for fifty years, it sprang back to life with the work of
M.~Jimbo, T.~Miwa, Y.~M\^ori, M.~Sato, T.~Ueno and other members of
the famous Kyoto school in the late seventies. Since then the theory
found a number of applications in statistical physics (see e.g. a
series of papers on holonomic quantum fields by the Kyoto school),
random matrix theory (see e.g. \cite{JMMS}, \cite{TW}, \cite{Pal},
\cite{HI}, \cite{BD}), theory of Frobenius manifolds (see
\cite{Du}), and representation theory (see \cite{BD}).

A central role in the theory of isomonodromy deformations is played
by the so-called {\it $\tau$-function} --- a holomorphic function on
the universal covering space of the space of parameters of the
connection, which vanishes when the corresponding isomonodromy
deformation fails to exist. The isomonodromy $\tau$-function was
initially introduced and studied by M.~Jimbo, T.~Miwa, and T.~Ueno
in \cite{JMU}, \cite{JMII}, \cite{JMIII}. It found various
interpretations in applications; e.g. in random matrix theory the
$\tau$-function appears as the {\it gap probability\/} --- the
probability that no eigenvalues of the random matrix are present in
a union of intervals. This fact can be seen as one reason why the
gap probabilities for one-interval gaps are often expressible
through solutions of the classical Painlev\'e equations, see e.g.
\cite{JMMS}, \cite{Me}, \cite{TW}, \cite{AvM1}, \cite{BD},
\cite{HI}, \cite{FW1}-\cite{FW4} for details.

The theory of isomonodromy transformations of {\it difference\/}
rational connections (d-connections, for short) on $\p1$ is much
younger. It was suggested by one of the authors in \cite{B2} and
employed in \cite{B1}, \cite{BB}, \cite{K}, \cite{AB}, \cite{Sa2}.
There are presently two principal applications of the theory: On the
one hand, isomonodromy transformations of d-connections with few
singularities provide a key for understanding the geometry of
discrete Painlev\'e equations from Sakai's hierarchy (see \cite{Sa1}
for the hierarchy and \cite{AB}, \cite{Sa2} for explicit
connections). On the other hand, the discrete isomonodromy
transformations can be used to compute the gap probabilities in
various discrete probabilistic models of random matrix type, see
\cite{B1}, \cite{BB}.

The main goal of this paper is to introduce the notion of the
$\tau$-function of a rational d-connection and its isomonodromy
transformations. We also show that in a continuous limit our
$\tau$-function agrees with the conventional one; we compute the
$\tau$-function for the isomonodromy transformations leading to
difference Painlev\'e V ($dPV$) and difference Painlev\'e VI
$(dPVI)$ equations (in the terminology of \cite{AB}), and we prove
that the gap probability for a wide class of discrete random matrix
type models can be viewed as the $\tau$-function for an associated
d-connection.

Let us describe our results in more detail.

Let $\vL$ be a vector bundle on $\p1$ of rank $m$. Define a
one-dimensional vector space $\detrg(\vL)$ by
\begin{equation*}
\detrg(\vL)=\det(H^0(\p1,\vL))\otimes(\det(H^1(\p1,\vL)))^{-1}.
\end{equation*}
Recall that $H^0(\p1,\vL)$ is the space of global regular sections
of $\vL$, and $H^1(\p1,\vL)$ can be interpreted as the space of
obstructions for a Mittag-Leffler problem, see Section
\ref{sc:detrg} for details. Both $H^0(\p1,\vL)$ and $H^1(\p1,\vL)$
are finite-dimensional.

In a sense, $\detrg(\vL)$ is the only nontrivial way to associate to
a vector bundle $\vL$ a one-dimensional vector space. More
precisely, we can view $\detrg$ as a line bundle on the moduli space
of vector bundles on $\p1$, and any other line bundle is its tensor
power (see \cite{LS} and references therein for the statement and
its generalizations).

The definition of $\detrg$ makes sense (and is widely used) in a
much more general situation; one description can be found in
\cite{KM}.

Observe that if $\vL\simeq (\vO(-1))^m$ then
$H^0(\p1,\vL)=H^1(\p1,\vL)=0$, so $\detrg(\vL)=\C$ and
$\detrg(\vL)^{-1}=\C$. In particular, there is a canonical element
$1\in\detrg(\vL)^{-1}$.

\begin{definition*} Suppose $\vL$ has
slope $-1$; that is, $\deg(\vL)=-m$. We define
$\tau(\vL)\in\detrg(\vL)^{-1}$ by
\begin{equation*}
\tau(\vL)=\begin{cases}1\quad\text{if }\vL\simeq(\vO(-1))^m,\cr
0\quad\text{otherwise.}\end{cases}
\end{equation*}
\end{definition*}

By itself, the element $\tau(\vL)\in\detrg(\vL)^{-1}$ provides
almost no meaningful information. However, if $\vL$ is equipped with
an additional structure, the derivatives of $\tau$ might be
meaningful. More precisely, given a d-connection on $\vL$ (of a
certain kind), we have a sequence of `modifications'
$\{\vL_n\}_{n\in\Z}$ and a canonical isomorphism
$\detrg(\vL_{n+1})\iso\detrg(\vL_n)\otimes S$, where $\vL_n$ is a
vector bundle on $\p1$, $\vL_0=\vL$, and  $S$ is a one-dimensional
vector space that does not depend on $n$. Therefore, the first ratio
$\tau(\vL_{n+1})/\tau(\vL_n)$ is
 a functional on $S$, while the second ratio
$$\frac{\tau(\vL_n)\tau(\vL_{n+2})}{\tau^2(\vL_{n+1})}$$
is a number (assuming that the denominator is nonzero).

All the modifications $\vL_n$ are equipped with d-connections, which
can be viewed as `isomonodromy transformations' of the initial
d-connection on $\vL_0$. Explanations of the term can be found in
\cite{B2}, \cite{K}.

This paper is organized as follows.

 Sections \ref{sc:mods} contains general definitions.

In Section \ref{sc:coordinates} we explicitly compute the ratios of
the $\tau$-function for modifications of three types: when two
simple zeroes of $\vA(z)$ shift in different directions, when a
simple zero and a simple pole shift in the same direction, and when
the shifting simple zero and simple pole coalesce. Note that these
are the simplest modifications that preserve the degree of $\vL$.

In Section \ref{sc:limit} we consider a limit transition that turns
a d-connection into an ordinary connection. We verify that the
second difference logarithmic derivatives of our $\tau$-function
converge to the second logarithmic derivatives of the conventional
isomonodromy $\tau$-function for the limiting connection, see
Theorem \ref{th:lim}. It is worth pointing out that in the
continuous situation the definition of the $\tau$-function
prescribes its first logarithmic derivatives rather than the second
ones. However, in the difference situation the first derivatives are
defined only up to a constant, and we were unable to find a natural
way to fix this constant.

In Section \ref{sc:Painleve} we compute the second ratios for
$\tau$-functions of isomonodromy transformations that reduce to
$dPV$ and $dPVI$ equations. The resulting expressions, see Theorems
\ref{th:taudPV}, \ref{th:taudPVI}, are surprisingly simple, and they
should be viewed as functions on the corresponding moduli spaces of
d-connections. The zeroes and poles of these second ratios show when
modifications of the corresponding d-connection lead to a nontrivial
vector bundle.

Section \ref{sc:probability} is dedicated to evaluating gap
probabilities for discrete biorthogonal random matrix type ensembles
associated with multiple orthogonal polynomials of mixed type in the
sense of \cite{DK}. This is a broad class of measures that naturally
appears in a variety of domains of mathematics including enumerative
combinatorics, tiling models, models of random growth, etc. In
Theorem \ref{th:det=tau} we prove that if the difference logarithmic
derivatives of all the relevant weight functions are rational, then
there exists a vector bundle with a rational d-connection such that
the first difference logarithmic derivatives of its $\tau$-function
(correctly defined because of certain explicit choices we make)
coincide with those of the gap probabilities for the biorthogonal
ensemble.

The final Section \ref{sc:Hahn} provides an example: We deal with
the Hahn orthogonal polynomial ensemble that comes up naturally in
the statistical description of tiling of a hexagon by rhombi (see
\cite{Joh}) and in harmonic analysis on the infinite-dimensional
unitary group (see \cite{BO}). Using the results of Sections
\ref{sc:Painleve} and \ref{sc:probability}, we show that the
one-interval gap probability for the Hahn ensemble is expressible
through a solution of the $dPVI$ equation, see Theorem
\ref{th:Hahn}. Even though a variety of results of this type are
known, see \cite{B1}, \cite{Ba}, \cite{AvM2}, \cite{FW1}-\cite{FW4},
\cite{BB}, it is the first time that such a result involves a
discrete Painlev\'e equation that is so high in Sakai's hierarchy.
Also, this is apparently the simplest example of a model that
probably cannot be handled using the isomonodromy deformations of
usual connections because $dPVI$ cannot be viewed as a symmetry of a
differential Painlev\'e equation.

\subsection*{Acknowledgements} The authors are very grateful to
Pierre Deligne for illuminating discussions. We are particularly
indebted to him for interpreting the $\tau$-function via determinant
of cohomology.

The second author (A.~B.) was partially supported by the NSF grant
DMS-0402047.

\section{Modifications and $\detrg$}\label{sc:mods}

\subsection{}
Let $\vL$ be a vector bundle on $\p1$ of rank $m$.

\begin{definition} A (rational) \emph{d-connection} on $\vL$ is a linear
operator
\begin{equation*}
\vA(z):\vL_z\to\vL_{z+1}
\end{equation*}
that depends on a point $z\in\p1-\{\infty\}$ in a rational way (in
particular, $\vA(z)$ is defined for all $z\in\C$ outside of a finite
set); here $\vL_z$ is the fiber of $\vL$ over $z\in\p1$. In other
words, $\vA$ is a rational map between the vector bundle $\vL$ and
its pullback via the automorphism $\p1\to\p1$ that sends $z\mapsto
z+1$. \label{df:dconn}
\end{definition}

\begin{definition}
We say that a point $z_0\in\p1$ is a \emph{pole} of $\vA$ if
$\vA(z)$ is not regular at $z=z_0$. We say that $z_0\in\p1$ is a
\emph{zero} of $\vA$ if the map
\begin{equation*}
\vA^{-1}(z):\vL_{z+1}\to\vL_z
\end{equation*}
is not regular at $z=z_0$. Note that $\vA$ can have a zero and a
pole at the same point.
\end{definition}

Denote by $\sing(\vA)\subset\C$ the set of all zeroes and poles of
$\vA$ on $\C=\p1-\{\infty\}$.

\begin{example} Suppose that $\vA$ has no pole at $x\in\C$ and that
$\det(\vA)$ has a simple zero at $x$. Obviously $x$ is a zero of
$\vA$. We will say that $x$ is a \emph{simple} zero.

Dually, suppose $\vA$ has no zero at $x\in\C$ and $\det(\vA)$ has
a simple pole at $x$. Then $x$ is a pole of $\vA$; we say that $x$
is a \emph{simple} pole.
\end{example}

\subsection{}

\begin{definition}
Suppose $\vR:\vL\isorat\vL'$ is a rational isomorphism between two
vector bundles $\vL$ and $\vL'$ on $\p1$. We call $\vL'$ a
\emph{modification} of $\vL$ (of course, $\vL$ is also a
modification of $\vL'$). If $\vR$ is a regular map (that is, it
has no poles, but it might have zeroes), we say that $\vL'$ is an
\emph{upper} modification of $\vL$ and $\vL$ is a \emph{lower}
modification of $\vL'$. For a fixed finite set $S\subset\p1$, we
say that $\vL'$ is a \emph{modification of $\vL$ on $S$} if
$\vR(z)$ and $\vR^{-1}(z)$ are regular outside of $S$.

A d-connection $\vA$ on $\vL$ induces a d-connection $\vA'$ on
$\vL'$. We call $\vA'$ a \emph{modification} of $\vA$.
\end{definition}

\begin{remark} Modifications can be viewed as an isomonodromy
deformation in the sense of \cite{B2}, see also \cite{K}. Indeed,
the monodromies of the difference equations associated with $\vA$
and $\vA'$ coincide (for the monodromies to exist, $\vA$ and $\vA'$
have to satisfy certain non-degeneracy conditions).
\end{remark}

\begin{example}  Suppose $\vR:\vL\isorat\vL'$ is regular and
$\det(\vR)$ has exactly one simple zero at $x\in\C$. In this case,
$\vL'$ is an \emph{elementary upper modification} (at $x$) of
$\vL$, and $\vL$ is an \emph{elementary lower modification} (at
$x$) of $\vL'$.

An elementary upper modification $\vR:\vL\to\vL'$ at $x$ is uniquely
determined by a one-dimensional subspace $l\subset\vL_x$ given by
$l=\ker(\vR(x):\vL_x\to\vL'_x)\subset\vL_x$. Conversely, any
one-dimensional $l\subset\vL_x$ defines an elementary upper
modification at $x$.

Dually, elementary lower modifications of $\vL'$ at $x$ are in
one-to-one correspondence with subspaces $l'\subset\vL'_x$ of
codimension one. (For $\vR:\vL\to\vL'$, we set
$l'=\im(\vR(x):\vL_x\to\vL'_x)$.) \label{ex:elementary}
\end{example}

\subsection{}
Let $\vA$ be a d-connection on $\vL$, and suppose
$x\in\sing(\vA)$. Then there exists a unique modification
$\vA^{\{x\}}$ of $\vA$ at $x$ such that $x$ is not a singular
point of $\vA^{\{x\}}$. Also, there exists a unique modification
$\vA_{\{x+1\}}$ of $\vA$ at $x+1$ such that $x$ is not a singular
point of $\vA_{\{x\}}$.

\begin{lemma} Suppose $x-1\not\in\sing(\vA)$.
\begin{enumerate}
\item $\sing(\vA^{\{x\}})=(\sing(\vA)\setminus\{x\})\cup\{x-1\}$;

\item $\vA$ is the unique modification of $\vA^{\{x\}}$ at $x$
with no singularity at $x-1$. That is,
$\vA=(\vA^{\{x\}})_{\{x\}}$.
\end{enumerate}
Dually, suppose $x+1\not\in\sing(\vA)$.
\begin{enumerate}
\setcounter{enumi}{2} \item
$\sing(\vA_{\{x+1\}})=(\sing(\vA)\setminus\{x\})\cup\{x+1\}$;

\item $\vA$ is the unique modification of $\vA_{\{x+1\}}$ at $x$
with no singularity at $x+1$. That is,
$\vA=(\vA_{\{x+1\}})^{\{x+1\}}$.
\end{enumerate}
\qed \label{lm:modif}
\end{lemma}

\begin{example}
If $\vA$ has a simple zero at $x$, then $\vA^{\{x\}}$ is an
elementary upper modification of $\vA$ at $x$. This modification
corresponds to the one-dimensional subspace
$l=\ker(\vA(x):\vL_x\to\vL_{x+1})$ in the sense of Example
\ref{ex:elementary}.

Dually, if $\vA$ has a simple pole at $x$, then $\vA^{\{x\}}$ is
the elementary lower modification of $\vA$ at $x$ corresponding to
the codimension one subspace $\im(\vA^{-1}(x):\vL_{x+1}\to\vL_x)$.
\label{ex:simplesingularity}
\end{example}

\subsection{}\label{sc:detrg} Consider cohomology spaces
$H^0(\p1,\vL)$ and $H^1(\p1,\vL)$. They have the following classical
interpretation: Fix any non-empty finite set $\frak S\subset\p1$.
Consider the (infinite dimensional) vector space
$\Gamma(\vL(\infty\cdot \frak S))$ of rational sections of $\vL$
that are allowed to have poles of any order at the points of $\frak
S $. Consider also the vector space of polar parts for rational
sections of $\vL$ at the points of $\frak S$. It is natural to
denote the space by $\Gamma(\vL(\infty\cdot \frak S)/\vL)$.

The natural linear map
\begin{equation*}
\Gamma(\vL(\infty\cdot \frak S))\to\Gamma(\vL(\infty\cdot \frak
S)/\vL)
\end{equation*}
sends a rational function to its polar part. The kernel of this
map is identified with the space $H^0(\p1,\vL)$ of global regular
sections of $\vL$. The cokernel is identified with $H^1(\p1,\vL)$;
this corresponds to the interpretation of classes in
$H^1(\p1,\vL)$ as obstructions for a Mittag-Leffler problem. Both
$H^0(\p1,\vL)$ and $H^1(\p1,\vL)$ are finite-dimensional.

\begin{notation*} For a finite-dimensional vector space $V$,
$\det(V)$ denotes the top exterior power of $V$. In particular,
$\det(0)=\C$. If $\dim(V)=1$, the dual of $V$ is denoted by
$V^{-1}$.
\end{notation*}

\begin{definition}
 Define a one-dimensional
vector space $\detrg(\vL)$ by
\begin{equation*}
\detrg(\vL)=\det(H^0(\p1,\vL))\otimes(\det(H^1(\p1,\vL)))^{-1}.
\end{equation*}
\end{definition}

For instance, suppose that $\vL\simeq (\vO(-1))^m$. Then
$H^0(\p1,\vL)=H^1(\p1,\vL)=0$, so $\detrg(\vL)=\C$ and
$\detrg(\vL)^{-1}=\C$. In particular, there is a canonical element
$1\in\detrg(\vL)^{-1}$.

\begin{definition} Suppose that a rank $m$ vector bundle $\vL$ has
slope $-1$; that is, $\deg(\vL)=-m$. We define
$\tau(\vL)\in\detrg(\vL)^{-1}$ by
\begin{equation*}
\tau(\vL)=\begin{cases}1\quad\text{if }\vL\simeq(\vO(-1))^m,\cr
0\quad\text{otherwise.}\end{cases}
\end{equation*}
\end{definition}

\begin{example} \label{ex:modification and detRG}
Suppose $\vL$ has slope $-1$. Let $\vL'$ be an arbitrary upper
modification of $\vL$. The quotient $\vL'/\vL$ is supported at
finitely many points (the zeroes of the map $\vL\to\vL'$) and its
space of global sections $H^0(\p1,\vL'/\vL)$ is finite-dimensional.
The long exact sequence corresponding to the sequence
\begin{equation*}
0\to\vL\to\vL'\to\vL'/\vL\to0
\end{equation*}
induces an identification between $\detrg(\vL)$ and
$\det(H^0(\p1,\vL'))\otimes\det(H^0(\p1,\vL'/\vL))^{-1}$.

Consider now the natural map
\begin{equation*}
q:H^0(\p1,\vL')\to H^0(\p1,\vL'/\vL).
\end{equation*}
Its determinant
$$\det(q)\in\det(H^0(\p1,\vL'/\vL))\otimes\det(H^0(\p1,\vL'))^{-1}=\detrg(\vL)^{-1}$$
equals $\tau(\vL)$. In particular, $q$ is an isomorphism if and only
if $\vL\simeq(\vO(-1))^m$.
\end{example}

\begin{example} Let $\vL_2$ be a modification of
$\vL_1$. Suppose that $\vL_2$ has slope $-1$ and
$\vL_1\simeq(\vO(-1))^m$. One can choose an upper modification
$\vL'$ of $\vL_1$ that is also an upper modification of $\vL_2$.
Consider the ratio
\begin{equation*}
\frac{\tau(\vL_2)}{\tau(\vL_1)}\in\detrg(\vL_1)\otimes\detrg(\vL_2)^{-1}.
\end{equation*}
By the previous example, we can identify
$$\detrg(\vL_1)\otimes\detrg(\vL_2)^{-1} \text{ and }
\det(H^0(\p1,\vL'/\vL_2))\otimes\det(H^0(\p1,\vL'/\vL_1))^{-1}.$$
Under this identification, the ratio corresponds to the
determinant of the composition
\begin{equation}
H^0(\p1,\vL'/\vL_1)\iso H^0(\p1,\vL')\iso H^0(\p1,\vL'/\vL_2).
\label{eq:comp}
\end{equation}
\label{ex:ratio}
\end{example}

\section{Ratios of the $\tau$-function}\label{sc:coordinates}

In this section, we study $\tau$-functions of vector bundles with
d-connections. First, we consider d-connections with arbitrary
singularities, and then look at three special cases.

\subsection{}\label{sc:ratios}
Let $\vL$ be a rank $m$ vector bundle and $\vA$ be a d-connection on
$\vL$. Suppose that $\vA$ has singularities at $n$ distinct points
$a_1,\dots,a_n$ and no singularities at $a_i+k$ for $k\in\Z-\{0\}$,
$i=1,\dots,n$. We impose no restrictions on the behavior of $\vA$
elsewhere.

Consider the lattice
\begin{equation*}
\Lambda\eqd(a_1,\dots,a_n)+\Z^n\subset\C^n.
\end{equation*}
Fix $u=(u_1,\dots,u_n)\in\Lambda$. By virtue of Lemma
\ref{lm:modif}, there exists a unique modification $\vL_u$ of $\vL$
at a subset of $\Lambda$ such that the d-connection $\vA_u$ on
$\vL_u$ satisfies
\begin{equation*}
(\sing(\vA_u)=\sing(\vA)\setminus\{a_1,\dots,a_n\})\cup\{u_1,\dots,u_n\}.
\end{equation*}
In other words, the singularities at $a_i$'s are shifted to $u_i$'s.

For $u\in\Lambda$,
$$\deg(\vL_u)=\deg(\vL)-\sum\kappa_i(u_i-a_i),$$
where $\kappa_i$ is the order of zero of $\det(\vA)$ at $a_i$
($\kappa_i$ can be negative). Consider the sublattice
\begin{equation*}
\Lambda_{-m}\eqd\{u\in\Lambda:\deg(\vL_u)=-m\}\subset\Lambda.
\end{equation*}

Set $T_u\eqd\detrg(\vL_u)^{-1}$. This is a one-dimensional vector
space depending on $u\in\Lambda$; if $u\in\Lambda_{-m}$, we have a
natural element $\tau(\vL_u)\in T_u$.

Note that according to our definition, $\tau$ is not a function on
$\Lambda_{-m}$, because its value belongs to a one-dimensional
vector space that has no natural basis. Nevertheless, it turns out
that the `second logarithmic derivative' of $\tau$ makes sense as a
function on $\Lambda_{-m}$. Let us make the statement precise:

\begin{proposition} For $i=1,\dots,n$, let $e_i\in\Z^n$ be the
$i$th standard basis vector. Then the `first derivative'
\begin{equation*}
S^{(i)}_u\eqd T_u\otimes(T_{u-e_i})^{-1}
\end{equation*}
does not depend on $u\in\Lambda$. That is, there exists a canonical
isomorphism $S^{(i)}_u\iso S^{(i)}_{v}$ for any $u,v\in\Lambda$.
\label{pp:derivative}
\end{proposition}

\begin{proof}
Fix $u=(u_1,\dots,u_n)\in\Lambda$ and $i=1,\dots,n$. The bundle
$\vL_{u-e_i}$ is a modification of $\vL_u$ at $u_i$; let us write it
as a combination of an upper modification $\vL_u\to\vL'_u$ and a
lower modification $\vL'_u\leftarrow\vL_{u-e_i}$ at $u_i$. As in
Example \ref{ex:ratio}, we have an isomorphism
\begin{equation}\label{eq:S}
S^{(i)}_u=\detrg(\vL_{u-e_i})\otimes\detrg(\vL)^{-1}=
\det(H^0(\p1,\vL'_u/\vL_u))\otimes\det(H^0(\p1,\vL'_u/\vL_{u-e_i}))^{-1}.
\end{equation}

From this description of $S^{(i)}_u$, one immediately gets an
identification $S^{(i)}_u\iso S^{(i)}_v$ provided $u$ and $v$ have
equal $i$th component. Indeed, $\vL_u$ and $\vL_v$ are naturally
identified in the neighborhood of $u_i=v_i$; denote the
identification by $\phi$. There is a unique upper modification
$\vL'_v$ of $\vL_v$ at $u_i$ such that $\phi$ identifies $\vL'_u$
and $\vL'_v$ near $u_i$. Then $\vL_{v-e_i}$ is a lower modification
of $\vL'_v$, and $\phi$ is also an isomorphism between $\vL_{u-e_i}$
and $\vL_{v-e_i}$ near $v_i$. Therefore, $\phi$ induces an
isomorphism between the right-hand side of \eqref{eq:S} and the
corresponding formula for $S^{(i)}_v$.

It remains to construct an isomorphism $S^{(i)}_u\iso
S^{(i)}_{u-e_i}$. For $z$ close to $u_i-1$, the map $\vA_u(z)$ is an
isomorphism between $\vL_{u-e_i}$ near $u_i-1$ and $\vL_u$ near
$u_i$. Indeed, $\vL_{u-e_i}$ coincides with $\vL_u$ near $u_i-1$,
while $\vA_u$ has no singularity at $z=u_i-1$. There is a unique
upper modification $\vL'_{u-e_i}$ of $\vL_{u-e_i}$ at $u_i-1$ such
that $\vA_u$ identifies $\vL'_{u-e_i}$ near $u_i-1$ with $\vL'_u$
near $u_i$. Then $\vL_{u-2e_i}$ is a lower modification of
$\vL'_{u-e_i}$, and $\vA_u$ is also an isomorphism between
$\vL_{u-2e_i}$ near $u_i-1$ and $\vL_{u-e_i}$ near $u_i$. Therefore,
$\vA_u$ induces an isomorphism between the right-hand side of
\eqref{eq:S} and the corresponding formula for $S^{(i)}_{u-e_i}$.

It is easy to see that the constructed isomorphisms $S^{(i)}_u\iso
S^{(i)}_v$ do not depend on the choice of the upper modification
$\vL_u\to\vL'_u$.
\end{proof}

Since $S^{(i)}_u$ does not depend on $u\in\Lambda$, we suppress the
index $u$ from now on.

Let us now define the ratios of $\tau$. Set
$$\Z^n_0=\{s=(s_1,\dots,s_n):\sum_{i=1}^n\kappa_is_i=0\}.$$ For any
$u\in\Lambda_{-m}$, $s\in\Z^n_0$, we have $u+s\in\Lambda_{-m}$.
Consider the ratio
$$D_s\tau(u)\eqd\frac{\tau(u+s)}{\tau(u)}\in T_{u+s}\otimes
T_u^{-1}=\bigotimes_i (S^{(i)})^{\otimes s_i}.$$ Now fixing
$t\in\Z^n_0$, we see that the second ratio
$$D_{s,t}\tau(u)\eqd\frac{D_s\tau(u+t)}{D_s(u)}$$
makes sense as a number.

\subsection{Sign issues}
\label{sc:signs} Let us choose bases in vector spaces $S^{(i)}$ for
$i=1,\dots,m$. Then for any $s\in\Z^n_0$, $u\in\Lambda_{-m}$, the
ratio $D_s\tau(u)$ becomes a number. It is tempting to say that
these numbers are partial derivatives of a function
\begin{equation*}
\tilde\tau:\Lambda_{-m}\to\C.
\end{equation*}
To construct such $\tilde\tau$, we need to choose bases in vector
spaces $T_u$ (for all $u\in\Lambda_{-m}$) consistent with the bases
in $S^{(i)}$ in the sense of Proposition \ref{pp:derivative}.

Generally speaking, this is impossible. More precisely, one can
choose bases in $T_u$ that are consistent up to sign. Equivalently,
$\tilde\tau$ can be defined as a function on a two-fold cover of
$\Lambda_{-m}$. Let us explain the sign in more details.

The basic reason for the sign is that for two finite-dimensional
vector spaces $V$, $W$, the isomorphism $\det(V\oplus
W)\simeq\det(V)\otimes\det(W)$ agrees with permutation up to sign
only. Indeed, the composition
\begin{equation*}
\det(V)\otimes\det(W)\simeq\det(V\oplus W)\simeq\det(W\oplus
V)\simeq\det(W)\otimes\det(V)\simeq\det(V)\otimes\det(W)
\end{equation*}
equals $(-1)^{\dim(V)\dim(W)}$. As a result, when we use Proposition
\ref{pp:derivative} to identify
$$
T_u\otimes(T_{u-e_i-e_j})^{-1}=(T_u\otimes(T_{u-e_i})^{-1})\otimes(T_{u-e_i}\otimes(T_{u-e_i-e_j})^{-1})=S^{(i)}\otimes
S^{(j)},
$$
the identification is multiplied by $(-1)^{\kappa_i\kappa_j}$ when
$i$ and $j$ are permuted (assuming $i\ne j$).

Denote by $\widetilde\Z^n$ the group generated by $\tilde e_i$
($i=1,\dots,n$) and $\epsilon$ subject to relations
\begin{equation*}
2\cdot \epsilon=0,\quad \epsilon\dotplus\tilde e_i=\tilde
e_i\dotplus\epsilon,\quad\tilde e_i\dotplus\tilde
e_j=(\kappa_i\kappa_j)\epsilon\dotplus\tilde e_j\dotplus\tilde e_i.
\end{equation*}
Recall that $\kappa_i$ is the order of zero of $\det(\vA)$ at $a_i$.
Consider the homomorphism $\pi:\widetilde\Z^n\to\Z^n$ that sends
$\tilde e_i$ to $e_i$ and $\epsilon$ to $0$. Using $\pi$, we can
view $\widetilde\Z^n$ as a central extension of $\Z^n$ by
$\{0,\epsilon\}$. The group $\Z^n$ acts on $\Lambda$; the action
lifts to a natural action of $\widetilde\Z^n$ on
$\{T_u\}_{u\in\Lambda}$. That is, for $u,v\in\Lambda$, an
isomorphism $T_u\iso T_v$ is determined by an element of
$\pi^{-1}(v-u)\subset\widetilde\Z^n$.

Set $\widetilde\Z^n_0\eqd\pi^{-1}(\Z^n_0)$. Fix a basis in $T_u$ for
single $u\in\Lambda_{-m}$. Acting by $\tilde s\in\widetilde\Z^n_0$,
we obtain a basis in $T_{u+\pi(\tilde s)}$. Therefore, $\tilde\tau$
is well defined as a function on the set
\begin{equation*}
\widetilde\Lambda_{-m}\eqd\{u\dotplus\tilde s:\tilde
s\in\widetilde\Z^n_0\}.
\end{equation*}

\begin{remark*} The situation simplifies if all $\kappa_i$'s are
even. In this case, the central extension $\widetilde\Z^n$ splits,
and $\tilde\tau$ makes sense as a function on $\Lambda_{-m}$. An
example of such situation is considered in Section
\ref{sc:zeroandpole}.
\end{remark*}

\subsection{d-connections with simple zeroes}\label{sc:zero}
Suppose that $\vA$ has a simple zero at $z=a_i$ for $i=1,\dots,n$.
Let us make the construction of Section \ref{sc:ratios} more
explicit in this case.

For $u\in\Lambda$, $i=1,\dots,n$, $\vL_{u-e_i}$ is an elementary
upper modification of $\vL_u$ at $u_i$.  The modification is given
by a dimension one subspace $l=l_{u,i}$ in the fiber of $\vL_u$ at
$u_i$ (see Example \ref{ex:elementary}). By Example
\ref{ex:simplesingularity}, $l$ is the kernel of the operator
\begin{equation*}
\vA_u(u_i):(\vL_u)_{u_i}\to(\vL_u)_{u_i+1}.
\end{equation*}

As in the proof of Proposition \ref{pp:derivative}, we have an
isomorphism
\begin{equation}
l=\detrg(\vL_{u-e_i})\otimes\detrg(\vL_u)^{-1}=S^{(i)}_u
\label{eq:l}
\end{equation}
(this corresponds to taking $\vL'_u=\vL_{u-e_i}$).

Set $s=e_i-e_j\in\Z^0_n$, and let us give explicit formulas for the
ratio $$D_s\tau(u)=\frac{\tau(u+e_i-e_j)}{\tau(u)}\in
S^{(i)}\otimes(S^{(j)})^{-1}.$$

Fix $u\in\Lambda_{-m}$, and let $l=l_{u,j}\subset(\vL_u)_{u_j}$ be
the dimension one subspace corresponding to the upper modification
$\vL_{u-e_j}$ of $\vL$. By \eqref{eq:l}, $S^{(j)}=l$.

Similarly, let us consider $\vL_{u+e_i}$ as a lower modification of
$\vL_u$ (instead of viewing $\vL_u$ as an upper modification of
$\vL_{u+e_i}$). Let $l'=l'_{u,i}\subset(\vL_u)_{u_i+1}$ be the
codimension one subspace corresponding to this modification. By the
proof of Proposition \ref{pp:derivative}, we get an identification
$S^{(i)}=(\vL_u)_{u_i+1}/l'$.

The identifications $S^{(i)}=l_{u+e_i,i}$ and
$S^{(i)}=(\vL_u)_{u_i+1}/l'$ are related as follows. Consider
multiplication by $z-(u_i+1)$; it gives a morphism
$\vL_u\to\vL_{u+e_i}$ (with a pole at infinity). Taking the value of
this morphism at $z=u_i+1$, we obtain a linear operator
$(\vL_u)_{u_i+1}\to(\vL_{u+e_i})_{u_i+1}$. It is easy to see that
the operator factors into a composition
\begin{equation}
(\vL_u)_{u_i+1}\to(\vL_u)_{u_i+1}/l'\iso l_{u+e_i,i}\hookrightarrow
(\vL_{u+e_i})_u. \label{eq:ltol'}
\end{equation}
This provides an isomorphism between $l_{u+e_i,i}$ and
$(\vL_u)_{u_i+1}/l'$.

The isomorphism \eqref{eq:ltol'} becomes more natural if we use the
d-connection to identify $l_{u,i}$ with $l_{u+e_i,i}$ (as in
Proposition \ref{pp:derivative}). Specifically, we obtain an
isomorphism
\begin{equation}
l_{u,i}\iso(\vL_u)_{u_i+1}/l'\quad:\quad
w\mapsto\frac{d\vA_u(u_i)}{dz}w. \label{eq:ltol'2}
\end{equation}

Suppose now that $\vL_u\simeq(\vO(-1))^m$. Fixing an isomorphism
$\iota:\vL_u\simeq(\vO(-1))^m$, we can identify the fiber
$(\vL_u)_z$ with $(\vO(-1))^m_z=\C^m$ for any point $z\in\C$. In
particular, the fibers $(\vL_u)_{u_j}$ and $(\vL_u)_{u_i+1}$ are
identified. Actually, this identification does not depend on the
choice of isomorphism $\iota$.

\begin{proposition}
The composition
\begin{equation*}
S^{(j)}=l\hookrightarrow(\vL_u)_{u_j}\iso(\vL_u)_{u_i+1}\to(\vL_u)_{u_i+1}/l'=S^{(i)}
\end{equation*}
is equal to multiplication by
\begin{equation*}
(u_i+1-u_j)\frac{\tau(u+e_i-e_j)}{\tau(u)}=(u_i+1-u_j)D_s\tau(u)
\end{equation*}
\label{pp:tauderivative}
\end{proposition}

\begin{proof} By Example \ref{ex:ratio}, the first derivative
$D_s\tau(u)$ is equal to the determinant of the composition
\begin{equation}
H^0(\p1,\vL_{u-e_j}/\vL_u)\iso H^0(\p1,\vL_{u-e_j})\iso
H^0(\p1,\vL_{u-e_j}/\vL_{u+e_i-e_j}).
\end{equation}

We can identify $H^0(\p1,\vL_{u-e_j}/\vL_u)$ with $l$ and
$H^0(\p1,\vL_{u-e_j}/\vL_{u+e_i-e_j})$ with
$(\vL_{u-e_j})_{u_i+1}/l'_{u-e_j,i}=(\vL_u)_{u_i+1}/l'$, where the
last equality follows from the fact that $\vL_u$ and $\vL_{u-e_j}$
coincide near $u_i+1$. Finally, global sections of $\vL_{u-e_j}$ are
of the form
\begin{equation*}
\frac{\lambda}{z-u_j}\,,\qquad \lambda\in l\subset(\vL_u)_{u_j}=\C^m
\end{equation*}
(we are using the identification $\iota:\vL_u\simeq(\vO(-1))^m$
here).
\end{proof}

Let us rewrite Proposition \ref{pp:tauderivative} using explicit
coordinates. Fix an isomorphism $\iota:\vL_u\simeq(\vO(-1))^m$. The
d-connection $\vA_u$ is then given by its matrix $A_u(z)$. By
assumption, $A_u(z)$ is regular at all points $a_i+\Z$,
$i=1,\dots,n$. Also, $\det(A_u(z))$ has simple zeroes at $u_i$ and
no zeroes at $u_i+(\Z-\{0\})$, $i=1,\dots,n$.

Using $\iota$, we identify $l_{u,j}\subset(\vL_u)_{u_j}$ with the
kernel of $A_u(u_j)$. Similarly, $l'_{u,i}\subset(\vL_u)_{u_i+1}$ is
identified with the image of $A_u(u_i)$. Instead of working with the
codimension one subspace $l'_{u,i}\subset\C^n$, we can consider its
orthogonal complement, which is a one-dimensional subspace
$(l'_{u,i})^\perp\subset(\C^n)^*=\C^n$.

Pick bases $w=w_{u,j}\in l_{u,j}$ and
$w'=w'_{u,i}\in(l'_{u,i})^\perp$. We have the following formula for
$A_{u+e_i-e_j}(z)$:
\begin{equation}
\begin{gathered}
A_{u+e_i-e_j}(z)=R_{u,i,j}(z+1)A_u(z)R_{u,i,j}(z)^{-1},\\
R_{u,i,j}(z)=I+\frac{R_0}{z-u_i-1}\,,\quad
R_{u,i,j}(z)^{-1}=I-\frac{R_0}{z-u_j}\,,\quad
\det R_{u,i,j}(z)=\frac{z-u_j}{z-u_i-1}\,,\\
R_0=\frac{u_i-u_j+1}{\langle w,w'\rangle}\,w\cdot(w')^t.
\end{gathered}
\label{eq:deformation}
\end{equation}

Observe that $A_{u+e_i-e_j}$ and $R_{u,i,j}$ are independent of the
choice of $w$ and $w'$.

In Proposition \ref{pp:derivative}, we show that for fixed $k$, the
vector space $l_{u,k}$ does not depend on $u\in\Lambda_{-m}$ in the
sense that the corresponding spaces are related by natural
isomorphisms. The isomorphisms $l_{u,k}\iso l_{u+e_i-e_j,k}$ are
given by the following formulas:
\begin{equation}
w\mapsto \begin{cases}R_{u,i,j}(u_k)\cdot w,\quad&k\ne i,j;\\
R_{u,i,j}(u_j-1)A_u(u_j-1)^{-1}\cdot w,\quad&k=j;\\
R_{u,i,j}(u_i+1)^{-1}A_{u+e_i-e_j}(u_i)\cdot w,\quad&k=i.
\end{cases}
\label{eq:basisshift}
\end{equation}
It is not hard to check explicitly that the isomorphisms are
consistent. Therefore, a choice of a basis $w\in l_{u,k}$ for one
$u\in\Lambda_{-m}$ determines bases $w_{v,k}\in l_{v,k}$ for all
$v\in\Lambda_{-m}$. Let us fix these bases.

Dually, the vector spaces $(l'_{u,k})^\perp$ are identified for all
$u\in\Lambda_{-m}$. The identifications are given by formulas
similar to \eqref{eq:basisshift}, which can be obtained using
\eqref{eq:ltol'2}. Denote by $w'_{u,k}\in l'_{u,k}$ the basis dual
to $w_{u,k}$.

We can now rewrite Proposition \ref{pp:tauderivative} as the
following formula:
$$\frac{\tilde\tau(u\dotplus\tilde e_i\dotminus\tilde e_j)}{\tilde\tau(u)}=\frac{\langle
w_{u,j},w'_{u,i}\rangle}{u_i+1-u_j}.
$$
Clearly, the second derivatives of $\tilde\tau$ are independent of
all choices.

These formulas can be used in a more `classical' definition of the
$\tau$ function as a solution to a system of difference equations.
From this point of view, the existence of a solution is not
obvious; this leads to the following statement.

\begin{corollary}
Let $A(z)$ be a square matrix with rational entries that is
regular at points $a_i+\Z$, $i=1,\dots,n$. Assume that
$\det(A(z))$ has simple zeroes at $a_i$ and no zeroes at
$a_i+(\Z-\{0\})$, $i=1,\dots,n$.

For any $u=(u_1,\dots,u_n)\in(a_1,\dots,a_n)+\Z^n$ with $\sum
u_i=\sum a_i$, we define the isomonodromy deformation $A_u(z)$
recursively using formulas \eqref{eq:deformation}.
\begin{enumerate}
\item $A_u(z)$ is well defined, provided $A(z)$ is generic in the
sense that $\langle w,w'\rangle$ does not vanish in
\eqref{eq:deformation}. In particular, $A_u(z)$ does not depend on a
representation of $u-a$ as a linear combination of generators
$e_i-e_j$.

\item Choose bases $w_{u,k}$, $w'_{u,k}$ as above. There exists a
function $\tilde\tau(a\dotplus \tilde s)$, where $\tilde
s\in\widetilde\Z^n_0$, such that
$$\frac{\tilde\tau(\tilde u\dotplus\tilde e_i\dotminus\tilde e_j)}
{\tilde\tau(\tilde u)}=\frac{\langle
w_{u,j},w'_{u,i}\rangle}{u_i+1-u_j}
$$
for every $\tilde u=a\dotplus \tilde s$, $\tilde
u\in\widetilde\Z^n_0$. Here $u=\pi(\tilde u)=a+\pi(\tilde s)$.
\end{enumerate}
\label{co:toappear}
\end{corollary}
\begin{proof}
(1) follows from Lemma \ref{lm:modif}; (2) follows from
Proposition \ref{pp:tauderivative}.
\end{proof}

\begin{remark*} Corollary \ref{co:toappear}(1) is a version of Theorem 2.1 of
\cite{B2}.
\end{remark*}

\subsection{d-connections with simple zeroes and poles}
\label{sc:zeroesandpoles} Now suppose the d-connection has simple
zeroes and simple poles. Let us consider isomonodromy deformations
corresponding to a simultaneous shift of a simple pole and a simple
zero. This kind of equations is important because it has one of the
simplest continuous limits (Section \ref{sc:limit}). Let us give an
analog of Corollary \ref{co:toappear}; the proof is completely
similar. We omit the coordinate-free formulation (an analog of
Proposition \ref{pp:tauderivative}).

Suppose $A(z)$ has simple zeroes at $n_a$ distinct points
$a_1,\dots,a_{n_a}$, simple poles at $n_b$ distinct points
$b_1,\dots,b_{n_b}$, and no singularities at $a_i+(\Z-\{0\})$,
$b_j+(\Z-\{0\})$. Suppose
\begin{equation*}
(u;v)=(u_1,\dots,u_{n_a};v_1,\dots,v_{n_b})\in(a_1,\dots,a_{n_a};b_1,\dots,b_{n_b})+\Z^{n_a+n_b}
\end{equation*}
satisfies $$\sum(u_i-a_i)=\sum(v_j-b_j).$$ Assuming $A(z)$ is
generic, there is a unique matrix $R(z)$ with rational
coefficients that satisfies the following conditions:
\begin{enumerate}
\item All singularities of $R(z)$ and $R^{-1}(z)$ belong to the
progressions $a_i+\Z$, $b_j+\Z$;

\item $R(\infty)=I$;

\item $A_{u;v}(z)=R(z+1)A(z)R(z)^{-1}$ has simple zeroes at
$u_1,\dots,u_{n_a}$, simple poles at $v_1,\dots,v_{n_b}$, and no
singularities at $u_i+(\Z-\{0\})$, $v_j+(\Z-\{0\})$.
\end{enumerate}

For instance, if $u=(a_1-1,a_2,\dots,a_{n_a})$ and
$v=(b_1-1,b_2,\dots,b_{n_b})$, $R(z)$ is given by the following
formulas:
\begin{equation}
\begin{gathered}
R(z)=I+\frac{R_0}{z-b_1}\,,\quad
R(z)^{-1}=I-\frac{R_0}{z-a_1}\,,\quad
\det R(z)=\frac{z-a_1}{z-b_1}\,,\\
R_0=\frac{b_1-a_1}{\langle w,w'\rangle}\,w\cdot(w')^t.
\end{gathered}
\end{equation}
Here $w$ is a basis in the kernel of $A(a_1)$, and $w'$ is a basis
in the image of $\lim_{z\to b_1}(z-b_1)A^t(z)$. Similar formulas can
be found for other `elementary shifts'. One can then use these
formulas to compute $A_{u;v}(z)$ recursively.

Similarly to \eqref{eq:basisshift}, a choice of $w$ and $w'$ for all
singular points of $A(z)$ determines bases in the corresponding
spaces for all deformations $A_{u;v}(z)$. This allows us to consider
$\tau$ as a function of $(u;v)$.

Similarly to Section \ref{sc:signs}, the function $\tau$ is defined
only up to a sign (that is, it is a function on the two-fold cover
of the set of $(u;v)$). One way to avoid this complication is to
assume that $n_a=n_b=n$ and that we always move $i$th zero and $i$th
pole simultaneously. Let us make this assumption, so that
$u_i-a_i=v_i-b_i$ for all $i$. Then the function $\tau$ satisfies
the following equation:
\begin{equation*}
\frac{\tau(u_1-1,u_2,\dots,u_n;v_1-1,v_2,\dots,v_n)}
{\tau(u_1,u_2,\dots,u_n;v_1,v_2,\dots,v_n)}=\frac{\langle
w,w'\rangle}{u_1-v_1}. \label{eq:D1tau}
\end{equation*}
Here $w$ is the basis in the kernel of $A_{u;v}(u_1)$, and $w'$ is
the basis in the image of $\lim_{z\to v_1}(z-v_1)A_{u;v}^t(z)$. As
above, the second derivatives are independent of the choice of $w$
and $w'$.

\subsection{}\label{sc:zeroandpole}
Let us now look at another type of singularity structure that
becomes useful in Section \ref{sc:probability}. Suppose $A(z)$ has
singularities at $n$ points $a_1,\dots,a_n$, and no singularities at
$a_i+(\Z-\{0\})$. The singularities at $a_i$ are of the following
kind:
\begin{itemize}
\item Matrix elements of $A(z)$ have at most a first-order pole at
$a_i$;
\item $\res_{a_i}(A(z))$ is a matrix of rank $1$;
\item $\det(A(z))$ is regular nonzero at $a_i$.
\end{itemize}
This can be viewed as a degeneration of the situation considered in
the previous section, when zeroes and poles coalesce.

If $A(z)$ is generic, for $(u_1,\dots,u_n)\in(a_1,\dots,a_n)+\Z^n$,
there exists a unique rational matrix $R(z)=R_u(z)$ with the
following properties:
\begin{enumerate}
\item All singularities of $R(z)$ and $R^{-1}(z)$ belong to the
progressions $a_i+\Z$;
\item $R(\infty)=I$;
\item $A_u(z)=R(z+1)A(z)R(z)^{-1}$ has the same singularity structure as
$A(z)$ with singularities at $u_i$.
\end{enumerate}

Choose a basis $w=w_u$ in the image of $\res_{u_1}(A^{-1}_u(z))$,
and a functional $w'+w''(z-u_1)=w'_u+w''_u(z-u_1)$ such that
\begin{gather}
\langle w, w'\rangle=0\label{cn:nopole}\\
w'\ne 0\\
A_u^{-t}(z)(w'+w''(z-u_1)) \text{ vanishes at
}z=u_1.\label{cn:vanishes}
\end{gather} Note that \eqref{cn:nopole} implies that
\eqref{cn:vanishes} is regular at $z=u_1$. Equivalently,
\eqref{cn:nopole}--\eqref{cn:vanishes} mean that in a neighborhood
of $z=u_1$, we can write
$$A(z)=H(z)\left(I+\frac{1}{z-u_1}\cdot\frac{w\cdot(w')^t}
{\langle w,w''\rangle}\right)$$ for a holomorphic invertible matrix
$H(z)$. The pair $(w',w'' \mod w^\perp)$ is defined up to a scalar.

\begin{remark*} Geometrically, the choices can be explained in terms
of Section \ref{sc:ratios}. Suppose the d-connection $\vA_u$ on a
vector bundle $\vL_u$ has at $z=u_1$ a singularity of the kind we
consider. There exists a unique elementary upper modification
$\vL'_u$ of $\vL_u$ at $u_1$ such that $\vL_{u-e_1}$ is an
elementary lower modification of $\vL'_u$ at $u_1$. The vector $w$
is a basis in the dimension one space $l\subset(\vL_u)_{u_1}$
corresponding to the modification $\vL_u\to\vL'_u$, while
$w'+w''(z-u_1)$ should be thought of as a functional on the fiber
$(\vL'_u)_{u_1}$ whose kernel is the codimension one space
$l'\subset(\vL'_u)_{u_1}$ corresponding to the modification
$\vL'_u\leftarrow\vL_{u-e_1}$.
\end{remark*}

For $u=(a_1-1,a_2,\dots,a_n)$, we can write $R(z)$ in terms of
$w,w',w''$:
\begin{equation*}
\begin{gathered}
R(z)=I+\frac{R_0}{z-a_1}\,,\quad
R(z)^{-1}=I-\frac{R_0}{z-a_1}\,,\quad
\det R(z)=1\,,\\
R_0=\frac{w\cdot(w')^t}{\langle w,w''\rangle}\,.
\end{gathered}
\end{equation*}

The choice of $(w,w',w'' \mod w^\perp)$ for $A(z)$ determines
corresponding choices $(w_u,w_u', w''_u \mod w^\perp_u)$ for all
deformations $A_u(z)$. Explicitly, if $u_1=a_1$, we have
\begin{gather*}
w_u=R_u(a_1)w,\qquad w_u'=R_u^{-t}(a_1)w',\\
w''_u=R_u^{-t}(a_1)w''+\left.\frac{dR_u^{-t}(z)}{dz}\right|_{z=a_1}w'.
\end{gather*}
Here $R_u(z)$ is the gauge matrix: $A_u(z)=R_u(z+1)A(z)R_u(z)^{-1}$.
On the other hand, for $u=(a_1-1,a_2,\dots,a_n)$, we have
\begin{gather*}
w_u=R_u(a_1-1)A^{-1}(a_1-1)w,\qquad w_u'=R_u^{-t}(a_1-1)A^{t}(a_1-1)w',\\
w''_u=R_u^{-t}(a_1-1)A^t(a_1-1)w''+\left.
\frac{d(R_u^{-t}(z)A^t(z))}{dz}\right|_{z=a_1-1}w'.
\end{gather*}

Similarly, we choose triples $(w,w',w''\mod w^\perp)$ at other
singularities $a_2,\dots,a_n$ of $A(z)$, and obtain corresponding
triples for all deformations $A_u(z)$. After these choices, we can
view $\tau$ as a function of $u$. Note that $\tau$ is defined
canonically (not just up to sign), because zeroes and poles move in
pairs (in terms of Section \ref{sc:ratios}, $\kappa_i=0$). The
equation for $\tau$ then becomes
$$
\frac{\tau(u_1-1,u_2,\dots,u_n)}{\tau(u_1,\dots,u_n)}=\langle
w_u,w_u''\rangle.
$$
(This is the value of functional $w_u'+w_u''(z-u_1)$ on section
$w/(z-u_1)$ of the elementary upper modification of the original
bundle.)

\subsection{Hirota identities}\label{sc:Hirota} $\tau$-functions satisfy various determinantal
identities of Hirota type. Let us show how they arise from
isomonodromy transformations. To be concrete, we restrict ourselves
to the case when $\vA$ has simple zeroes. Another case, when
singularities are of type considered in Section
\ref{sc:zeroandpole}, appears in Section \ref{sc:probability}, see
Remark \ref{rm:Hirota2}.

\begin{proposition} In the settings of Section \ref{sc:zero}, assume
that $\vL_u\simeq(\vO(-1))^m$. Then
\begin{equation}
\frac{\tau\left(u+\sum\limits_{i\in I} e_i-\sum\limits_{j\in
J}e_j\right)}{\tau(u)}=\det\left[\frac{\tau(u+e_i-e_j)}{\tau(u)}\right]_{i\in
I,j\in J}.\label{eq:Hirota}
\end{equation}
Here $I,J\subset\{1,\dots,n\}$ are non-intersecting subsets of the
same cardinality. \label{pp:Hirota}
\end{proposition}

\begin{remark*} As explained in Section \ref{sc:signs}, $\tau(u+s)/\tau(u)$ is
defined up to a sign, corresponding to the lift of $s\in\Z^m$ to the
two-fold cover $\widetilde\Z^m$. In Proposition \ref{pp:Hirota}, the
lift is chosen as follows: for orderings $i_1,\dots,i_k$ of $I$ and
$j_1,\dots,j_k$ of $J$, we take
$$\tilde e_{i_1}\dotminus\tilde e_{j_1}\dotplus\dots\dotplus
\tilde e_{i_k}\dotminus\tilde e_{j_k}$$ as the lift of $\sum
e_i-\sum e_j$ in the left-hand side, and $\tilde e_i\dotminus\tilde
e_j$ as the lift of $e_i-e_j$ in the right-hand side. The orderings
of $I$ and $J$ also fix the order of rows and columns in the
determinant.
\end{remark*}

\begin{proof}
Let us choose bases $w_{u,k}$, $w'_{u,k}$ as in Section
\ref{sc:zero}. By Example \ref{ex:ratio}, the left-hand side of
\eqref{eq:Hirota} equals the determinant of the transition matrix in
$H^0(\p1,\vL_{u-\sum e_j})$ from the basis consisting of meromorphic
sections of $\vL_u$ with a single pole at $u_j$ and residue
$w_{u,j}$ to the dual basis of the functionals that send a section
$s$ to $w'_{u,i}(s(u_i+1))$. (These bases come from
$H^0(\p1,\vL_{u-\sum e_j}/\vL)$ and $H^0(\p1,\vL_{u-\sum
e_j}/\vL_{u+\sum e_i-\sum e_j})$, respectively.)

Explicitly, we can choose a trivialization $\vL_u\simeq(\vO(-1))^m$,
and then
$$\frac{\tau\left(u+\sum\limits_{i\in I} e_i-\sum\limits_{j\in
J}e_j\right)}{\tau(u)}=\det\left[\frac{\langle
w_{u,j},w'_{u,i}\rangle}{u_i+1-u_j}\right]_{i\in I,j\in J}.
$$
The statement follows.
\end{proof}

\begin{remark*} Suppose $I=\{i_1,i_2\}$,
$J=\{j_1,j_2\}$. Then \eqref{eq:Hirota} takes the form
\begin{multline*}
\tilde\tau(u\dotplus e_{i_1} \dotminus e_{j_1}\dotplus
e_{i_2}\dotminus e_{j_2})\,\tilde\tau(u)\\= \tilde\tau(u\dotplus
e_{i_1}\dotminus e_{j_1})\,\tilde\tau(u\dotplus e_{i_2}\dotminus
e_{j_2})- \tilde\tau(u\dotplus e_{i_1}\dotminus
e_{j_2})\,\tilde\tau(u\dotplus e_{i_2}\dotminus e_{j_1}).
\end{multline*} This is often called Hirota's difference bilinear
equation.
\end{remark*}

\section{Continuous limit}\label{sc:limit}

\subsection{}
Let us recall some properties of isomonodromy deformation of
connections on $\p1$ in the simplest case of regular singularities.

Consider the system of linear ordinary differential equations
\begin{equation}
\frac{dY(\zeta)}{d\zeta}=B(\zeta)Y(\zeta),\quad
B(\zeta)=\sum_{i=1}^n\frac{B_i}{\zeta-y_i}\,, \label{eq:connection}
\end{equation}
where $B_i$'s are constant $m\times m$ matrices. Clearly,
\eqref{eq:connection} has regular singularities (simple poles) at
$\zeta=y_1,\dots,y_n,\infty$ and no other poles. Geometrically, we
can view \eqref{eq:connection} as a connection on the trivial vector
bundle on $\p1$.

The isomonodromy deformation of \eqref{eq:connection} is controlled
by a system of differential equations on $B_i$'s (viewed as
functions of $y_j$'s) called the \emph{Schlesinger system}:
\begin{equation}
\frac{\partial B_i}{\partial y_j}=\frac{[B_i,B_j]}{y_i-y_j}\,, \quad
\frac{\partial B_i}{\partial y_i}=-\sum_{j\ne
i}\frac{[B_i,B_j]}{y_i-y_j}\,. \label{eq:schlesinger}
\end{equation}

Instead of working with $B_i$, let us consider $B(\zeta)$ given by
\eqref{eq:connection}. Then \eqref{eq:schlesinger} can be written as
\begin{equation}
\frac{\partial B(\zeta)}{\partial
y_i}=\frac{B_i}{(\zeta-y_i)^2}-\left[\frac{B_i}{\zeta-y_i},B(\zeta)\right]\,.
\label{eq:varschlesinger}
\end{equation}

Let us now introduce the tau-function of the Schlesinger system. We
follow M.~Jimbo, T.~Miwa, and K.~Ueno \cite{JMU}.

For any solution of the Schlesinger system, the 1-form
\begin{equation*}
\omega=\sum_{i=1}^k\left(\sum_{j\ne
i}\frac{\tr(B_iB_j)}{y_i-y_j}\right)dy_i
\end{equation*}
 is closed. Locally, there exists a function $\tau$ with
$d\log(\tau)=\omega$.

Using Schlesinger system, one easily computes
\begin{equation}
\frac{\partial^2\log(\tau)}{\partial y_i\partial y_j}=\frac{\tr
(B_iB_j)}{(y_i-y_j)^2}\,,\qquad \frac{\partial^2\log(\tau)}{\partial
y_i^2}=-\sum_{j\ne i}\frac{\tr (B_iB_j)}{(y_i-y_j)^2}\,.
\label{eq:secondder}
\end{equation}

Our goal is to show how \eqref{eq:schlesinger} and
\eqref{eq:secondder} appear as limits of their discrete analogs.

\subsection{}
Return now to the setting of Section \ref{sc:zeroesandpoles}. Let
$\vA(z)$ be a d-connection on $(\vO(-1))^m$ with simple zeroes at
$n$ distinct points $a_1,\dots,a_n$, simple poles at $n$ distinct
points $b_1,\dots,b_n$, and no other singularities. Assume also that
$\vA(\infty)=I$. As usual, we assume that no two singularities
differ by an integer. We consider the action of $\Z^n$ by
isomonodromy transformations that shift $a_i$ and $b_i$
simultaneously.

For every $(u;v)=(u_1,\dots,u_n;v_1,\dots,v_n)$ such that
$u_i-a_i=v_i-b_i\in\Z$, denote the corresponding modification of
$\vA(z)$ by $\vA_{u;v}(z)$. The matrix of $\vA_{u;v}(z)$ is denoted
by $A_{u;v}(z)$. Also, for every index $i=1,\dots,n$, we choose a
basis $w_{u;v|i}$ of $\ker(A_{u;v}(u_i))$, and a basis $w'_{u;v|i}$
in the image of $\lim_{z\to v_i}(z-v_i)A_{u;v}^t(z)$. The choices
for different $(u;v)$ have to be compatible as described in Section
\ref{sc:zeroesandpoles}.

Let us introduce the following notation. For $(u;v)$ as above, set
\begin{align*}
D_i(\tau(u;v))&=\frac{\tau(u+e_i;v+e_i)}{\tau(u;v)}\,,&i=1,\dots,n;\\
D^2_{i,j}(\tau(u;v))&=\frac{\tau(u+e_i+e_j;v+e_i+e_j)\cdot\tau(u;v)}
{\tau(u+e_i;v+e_i)\cdot\tau(u+e_j;v+e_j)}\,,&i,j=1,\dots,n.
\end{align*}

\begin{theorem}
Assume that our data depend on the small parameter $\varepsilon\ne0$
so that
$$a_i=\alpha_i+\frac{y_i}{\varepsilon},\qquad
b_i=\beta_i+\frac{y_i}{\varepsilon},\qquad
\lim_{\varepsilon\to0}\frac{w_{a;b|i}(\varepsilon)\cdot
(w'_{a;b|i}(\varepsilon))^t}{\langle
w_{a;b|i}(\varepsilon),w'_{a;b|i}(\varepsilon)\rangle}=\frac{B_i}{\beta_i-\alpha_i}$$
for some matrices $B_i$.

Fix $(u;v)$ as above and $\zeta\in\C-\{y_1,\dots,y_n\}$. Then
\begin{gather}
A_{u;v}(\zeta\varepsilon^{-1};\varepsilon)=I+
\varepsilon\sum_{i=1}^n\frac{B_i}{\zeta-y_i}+o(\varepsilon);
\label{eq:lim1}\\
\lim_{\varepsilon\to0}
\frac{A_{u+e_i;v+e_i}(\zeta\varepsilon^{-1};\varepsilon)-
A_{u;v}(\zeta\varepsilon^{-1};\varepsilon)}{\varepsilon^2}=
\frac{B_i}{(\zeta-y_i)^2}-\sum_{j\ne
i}\frac{[B_i,B_j]}{(\zeta-y_i)(\zeta-y_j)}\,;
\label{eq:lim2}\\
\lim_{\varepsilon\to0}
\frac{D^2_{i,j}(\tau(u;v;\varepsilon))-1}{\varepsilon^2}=
\begin{cases}
\dfrac{\tr (B_iB_j)}{(y_i-y_j)^2},\quad &(i\ne j)\\
\\
-\sum\limits_{k\ne i}\dfrac{\tr (B_iB_k)}{(y_i-y_k)^2},\quad &(i=j).
\end{cases}
\label{eq:lim3}
\end{gather}
\label{th:lim}
\end{theorem}

The right-hand sides of \eqref{eq:lim1}, \eqref{eq:lim2}, and
\eqref{eq:lim3} correspond to \eqref{eq:connection},
\eqref{eq:varschlesinger}, and \eqref{eq:secondder}.

\begin{remark*} Note that Theorem \ref{th:lim} leads to the
Schlesinger system with rank one matrices $B_i$. One can obtain the
general case by a proper limiting procedure bringing several
singularities together.
\end{remark*}

\subsection{Proof of Theorem \ref{th:lim}}

\begin{lemma}
 Let $w_{a;b|i}(\varepsilon)$ and
$w'_{a;b|i}(\varepsilon)$ $(i=1,\dots,n)$ be vector-valued functions
of $\varepsilon\ne0$. Suppose that they satisfy the limit relation
above. For $\varepsilon$ small enough, there exists unique
$A_{a;b}(z;\varepsilon)$ of the kind we consider that corresponds to
these data.
\end{lemma}

\begin{proof}
Uniqueness of $A_{a;b}(z;\epsilon)$ is almost obvious, since a
rational matrix is determined by its singularity data and asymptotic
behavior at infinity.

Let us prove existence. Proceed by induction in $n$. Set
\begin{equation}
R_{a;b|i}(z;\varepsilon)=I+\frac{\beta_i-\alpha_i}{z-b_i}\cdot\frac{w_{a;b|i}(\varepsilon)\cdot
(w'_{a;b|i}(\varepsilon))^t}{\langle
w_{a;b|i}(\varepsilon),w'_{a;b|i}(\varepsilon)\rangle}. \label{eq:R}
\end{equation}
By the hypotheses,
\begin{equation}
R_{a;b|i}(\zeta\varepsilon^{-1};\varepsilon)=I+\varepsilon\frac{B_i}{\zeta-y_i}+o(\varepsilon).
\label{eq:asR}
\end{equation}

We then construct $A_{a;b}(z;\epsilon)$ as
\begin{equation}
A_{a;b}(z;\epsilon)=\tA_{a;b}(z;\epsilon)\cdot
R_{a;b|n}(z;\varepsilon), \label{eq:A}
\end{equation}
where $\tA_{a;b}(z;\epsilon)$ is a matrix-valued function such that
$\tA_{a;b}(\infty;\epsilon)=I$, $\tA_{a;b}(z;\epsilon)$ has simple
zeroes at $a_1,\dots,a_{n-1}$ and simple poles at
$b_1,\dots,b_{n-1}$ (and no other singularities), and
$\ker(\tA_{a;b}(a_i;\epsilon))$ (resp. image of $\lim_{z\to
b_i}(z-b_i)\tA^t_{u;v}(z;\epsilon)$) is spanned by
$R_{a;b|n}(a_i;\epsilon)w_{a;b|i}(\epsilon)$ (resp.
$R_{a;b|n}^{-t}(b_i;\epsilon)w_{a;b|i}'(\epsilon)$). Such $\tA$
exists by the induction hypothesis.
\end{proof}

Let us now prove Theorem \ref{th:lim}. Clearly, \eqref{eq:lim1}
follows from \eqref{eq:asR}, \eqref{eq:A}.

Let us prove \eqref{eq:lim2}. Without losing generality, we can
assume $(u;v)=(a-e_i,b-e_i)$. Then
$A_{u;v}(z;\epsilon)=R_{a;b|i}(z+1;\epsilon)A_{a;b}(z;\epsilon)R_{a;b|i}(z,\epsilon)^{-1}$,
where $R_{a;b|i}$ is given by \eqref{eq:R}. We then have
\begin{gather*}
A_{u+e_i;v+e_i}(\zeta\varepsilon^{-1};\varepsilon)-
A_{u;v}(\zeta\varepsilon^{-1};\varepsilon)\\=
A_{a;b}(\zeta\epsilon^{-1};\epsilon)-R_{a;b|i}(\zeta\epsilon^{-1}+1;\epsilon)A_{a;b}(\zeta\epsilon^{-1};\epsilon)
R_{a;b|i}(\zeta\epsilon^{-1};\epsilon)^{-1}\\=
\biggl[I-R_{a;b|i}(\zeta\epsilon^{-1}+1;\epsilon)R_{a;b|i}
\zeta\epsilon^{-1};\epsilon)^{-1}\biggr]\\-
\biggl[R_{a;b|i}(\zeta\epsilon^{-1}+1;\epsilon)(A_{a;b}(\zeta\epsilon^{-1};\epsilon)-I)
R_{a;b|i}(\zeta\epsilon^{-1};\epsilon)^{-1}-(A_{a;b}(\zeta\epsilon^{-1};\epsilon)-I)\biggr].
\end{gather*}
Using \eqref{eq:lim1} and \eqref{eq:R}, we see that the first
bracket divided by $\epsilon^2$ (resp. the second bracket divided by
$\epsilon^2$) converges to the first (resp. second) term in the
right-hand side of \eqref{eq:lim2}.

It remains to prove \eqref{eq:lim3}. By \eqref{eq:D1tau}, we have
\begin{gather*}
D_i(\tau(u;v;\epsilon))=\frac{\alpha_i-\beta_i}{\langle
w_{u+e_i;v+e_i|i}(\epsilon),w'_{u+e_i;v+e_i|i}(\epsilon)\rangle}\,,\\
D^2_{i,j}(\tau(u;v;\epsilon))=\frac{\langle
w_{u+e_i;v+e_i|i}(\epsilon),w'_{u+e_i;v+e_i|i}(\epsilon)\rangle}{\langle
w_{u+e_i+e_j;v+e_i+e_j|i}(\epsilon),w'_{u+e_i+e_j;v+e_i+e_j|i}(\epsilon)\rangle}\,.
\end{gather*}

Without loss of generality, we can assume that
$(u;v)=(a-e_i-e_j;b-e_i-e_j)$. For $i\ne j$, we have
$$w_{a-e_j;b-e_j|i}(\epsilon)=R_{a;b|j}(a_i;\epsilon)w_{a;b|i}(\epsilon),\qquad
w'_{a-e_j;b-e_j|i}(\epsilon)=R^{-t}_{a;b|j}(b_i;\epsilon)w'_{a;b|i}(\epsilon).
$$
Then $$D^2_{i,j}(\tau(u;v;\epsilon))=\frac{\langle
R_{a;b|j}(a_i;\epsilon)w_{a;b|i}(\epsilon),R^{-t}_{a;b|j}(b_i;\epsilon)w'_{a;b|i}(\epsilon)\rangle}
{\langle w_{a;b|i}(\epsilon),w'_{a;b|i}(\epsilon)\rangle}\,.
$$
The difference of the numerator and the denominator equals
$$\langle
(R_{a;b|j}(a_i;\epsilon)-R_{a;b|j}(b_i;\epsilon))
w_{a;b|i}(\epsilon),R^{-t}_{a;b|j}(b_i;\epsilon)w'_{a;b|i}(\epsilon)\rangle.$$
By \eqref{eq:R},
$$
R_{a;b|j}(a_i;\epsilon)-R_{a;b|j}(b_i;\epsilon)=\epsilon^2(\beta_i-\alpha_i)\frac{B_j}{(y_i-y_j)^2}+o(\epsilon^2).
$$
This implies the statement.

Finally, suppose $i=j$. Then (cf. \eqref{eq:basisshift})
\begin{gather*}
w_{a-e_i;b-e_i|i}(\epsilon)=
R_{a;b|i}(a_i-1;\epsilon)A^{-1}_{a;b}(a_i-1;\epsilon)w_{a;b|i}(\epsilon),
\\
w'_{a-e_i;b-e_i|i}(\epsilon)=
R^{-t}_{a;b|i}(b_i-1;\epsilon)A^{t}_{a;b}(b_i-1;\epsilon)w'_{a;b|i}(\epsilon).
\end{gather*}
The statement now follows from the asymptotics
\begin{multline*}
R_{a;b|i}(a_i-1;\epsilon)A^{-1}_{a;b}(a_i-1;\epsilon)-
R_{a;b|i}(b_i-1;\epsilon)A^{-1}_{a;b}(b_i-1;\epsilon)\\=-
\epsilon^2(\beta_i-\alpha_i)\sum_{j\ne
i}\frac{B_j}{(y_i-y_j)^2}+o(\epsilon^2).
\end{multline*}
\qed

\section{Discrete Painlev\'e equations}\label{sc:Painleve}

In some special cases, isomonodromy transformation gives rise to
discrete Painlev\'e equations (\cite{JS,B1,BB,AB,Sa2}). In this
section, we evaluate the $\tau$-function for the two cases
considered in \cite{AB}.

\subsection{Difference $P_V$ and difference $P_{VI}$}
Suppose $\vL$ is a rank $2$ vector bundle on $\p1$, and that the
d-connection $\vA$ has simple zeroes at $a_1$, $a_2$, simple poles
at $b_1$, $b_2$, and no other singularities. Also, fix the `formal
type' of $\vA$ at infinity: there exists a trivialization
$\vR(z):\C^2\to\vL_z$ on the formal neighborhood of infinity such
that the matrix of $\vA$ with respect to $\vR$ equals
$$\vR(z+1)^{-1}\vA(z)\vR(z)=\begin{bmatrix}\rho_1(1+\frac{d_1+b_1+b_2+1}{z})&0\\0&\rho_2(1+\frac{d_2+b_1+b_2+1}{z})\end{bmatrix},
$$
for $\rho_1,\rho_2,d_1,d_2\in\C$. (This choice of parameters is used
to match the formulas of \cite{AB}.) Finally, suppose that
$$d_1+d_2+a_1+a_2+b_1+b_2=0.$$ This implies that $\deg(\vL)=-2$.

Assuming the parameters are generic, the moduli space of such
d-connections is a surface (of type $D_4^{(1)}$), see \cite{AB}. Let
us introduce the coordinates on this surface.

For generic $(\vL,\vA)$, there exists an isomorphism
$\vL\iso(\vO(-1))^2$ such that the matrix of $\vA$ is of the form
$$
A(z)=\begin{bmatrix}a_{11}(z)&O(z)\\
z-q&\rho_2z^2+\rho_2d_2z+O(1)\end{bmatrix}\cdot\frac{1}{(z-b_1)(z-b_2)}\,,
$$
where $a_{11}(z)$ is of the form
$a_{11}(z)=\rho_1z^2+\rho_1d_1z+O(1)$. $A(z)$ is uniquely determined
by $q$ and $a_{11}(q)$; the other coefficients can be found using
the singularity structure of $A(z)$. We take the (rational)
coordinates on the moduli space to be $q$ and
$$p=\frac{a_{11}(q)}{(q-a_2)(q-b_2)}\,.$$

Consider the isomonodromy deformation that shifts $a_1\mapsto
a_1-1$, $b_1\mapsto b_1-1$. According to our choice of parameters,
it also shifts $d_1\mapsto d_1+1$, $d_2\mapsto d_2+1$, because the
formal type at the infinity does not change.

\begin{proposition}[{\cite[Theorem B]{AB}}] The transformed coordinates $q'$, $p'$ are related to
$p$ and $q$ by the following equations (difference $P_V$):
\begin{equation*}
\begin{cases}
q'+q=a_2+b_2+\dfrac{\rho_1(d_1+a_2+b_2)}{p-\rho_1}+
\dfrac{\rho_2(d_2+a_2+b_2+1)}{p-\rho_2}\,,
\\
p'p=\dfrac{(q'-a_1+1)(q'-b_1+1)}{(q'-a_2)(q'-b_2)}\cdot\rho_1\rho_2.
\end{cases}
\end{equation*}
\label{pp:dPV} \qed
\end{proposition}

Now consider d-connections with different singularity structure.
Namely, suppose $\vA$ has simple zeroes at $a_1$, $a_2$, $a_3$,
simple poles at $b_1$, $b_2$, $b_3$, and no other singularities.
Assume that in the formal neighborhood of infinity, there exists a
trivialization $\vR(z):\C^2\to\vL_z$ such that the matrix of $\vA$
with respect to $\vR$ equals
$$\vR(z+1)^{-1}\vA(z)\vR(z)=\begin{bmatrix}1+\frac{d_1+b_1+b_2+b_3+1}{z}&0\\0&1+\frac{d_2+b_1+b_2+b_3+1}{z}\end{bmatrix},
$$
for $d_1,d_2\in\C$. Finally, suppose that
$$d_1+d_2+a_1+a_2+a_3+b_1+b_2+b_3=0.$$ This implies that $\deg(\vL)=-2$.

For generic $(\vL,\vA)$, there exists an isomorphism
$\vL\iso(\vO(-1))^2$ such that the matrix of $\vA$ is of the form
$$
A(z)=\begin{bmatrix}a_{11}(z)&O(z)\\
z-q&z^3+d_2z^2+O(z)\end{bmatrix}\cdot\frac{1}{(z-b_1)(z-b_2)(z-b_3)}\,,
$$
where $a_{11}(z)$ is of the form $z^3+d_1z^2+O(z)$. Then $A(z)$ is
determined by $q$ and $a_{11}(q)$. It is more convenient to work in
coordinates $q$ and
$$r=\frac{(q-a_2)(q-a_3)(q-b_2)(q-b_3)}{a_{11}(q)}-q.$$

As above, consider the isomonodromy deformation that shifts
$a_1\mapsto a_1-1$, $b_1\mapsto b_1-1$; it also shifts $d_1\mapsto
d_1+1$, $d_2\mapsto d_2+1$.

\begin{proposition}[{\cite[Theorem F]{AB}}] The transformed coordinates $q'$, $p'$ are related to
$p$ and $q$ by the following equations (difference $P_{VI}$):
\begin{equation*}
\begin{cases}
(q+r)(q'+r) =\dfrac{(r+a_2)(r+a_3)(r+b_2)(r+b_3)}
{(r+1-a_1-b_1-d_1)(r-a_1-b_1-d_2)}\,,
\\
\\
(q'+r)(q'+r')= \dfrac{(q' - a_2) (q' - a_3) (q' - b_2) (q' -
b_3)}{(q' - (a_1-1)) (q'-(b_1-1))}\,.
\end{cases}
\end{equation*}
\label{pp:dPVI} \qed
\end{proposition}

\begin{remark*} These equations previously appeared in \cite{GRO} as the asymmetric dPIV
equation; see also references therein. The equivalence of
\cite[Theorem F]{AB} and the equations is explained in the
introduction to \cite{AB}.
\end{remark*}

\subsection{Tau-functions}
Following the recipe of Section \ref{sc:zeroesandpoles}, we can
write the second (logarithmic difference) derivative of the
tau-function in the direction of the above isomonodromy
transformations. The computations are somewhat tedious, but the
answer is remarkably simple:

\begin{theorem} In the settings of Proposition \ref{pp:dPV},
\begin{align*}
D^2\tau=\frac{\tau''\cdot\tau}{(\tau')^2}&=\frac{(p'-\rho_1)(\rho_1(q'-a_1+1)(q'-b_1+1)-p'(q'-a_2)(q'-b_2))}
{\rho_1(a_2-a_1+1)(b_2-b_1+1)p'}\\
&=\frac{(p'-\rho_1)(p-\rho_2)(q'-a_2)(q'-b_2)}
{\rho_1\rho_2(a_2-a_1+1)(b_2-b_1+1)}\,.
\end{align*}
Here $\tau'$ and $\tau''$ correspond to shifts
$(a_1,b_1)\mapsto(a_1-1,b_1-1)$ and $(a_1,b_1)\mapsto(a_1-2,b_1-2)$,
respectively. \qed \label{th:taudPV}
\end{theorem}

\begin{theorem} In the settings of Proposition \ref{pp:dPVI},
\begin{align*}
D^2\tau=\frac{\tau''\cdot\tau}{(\tau')^2}&=\frac{1}
{(a_1-a_2-1)(a_1-a_3-1)(b_1-b_2-1)(b_1-b_3-1)}\\
&\times \frac{r'-a_1-b_1-d_2+1}{q'+r'}\cdot
\bigl((q'-a_2)(q'-a_3)(q'-b_2)(q'-b_3)\\&-(q'-a_1+1)(q'-b_1+1)(q'+d_1+a_1+b_1-1)(q'+r')\bigr)\\
&=\frac{(r'-a_1-b_1-d_2+1)(r-a_1-b_1-d_1+1)(q'-a_1+1)(q'-b_1+1)}
{(a_1-a_2-1)(a_1-a_3-1)(b_1-b_2-1)(b_1-b_3-1)}\,.
\end{align*} Here $\tau'$ and $\tau''$ correspond to shifts
$(a_1,b_1)\mapsto(a_1-1,b_1-1)$ and $(a_1,b_1)\mapsto(a_1-2,b_1-2)$,
respectively. \qed \label{th:taudPVI}
\end{theorem}

\begin{remark*} Propositions \ref{pp:dPV}, \ref{pp:dPVI}, and
Theorems \ref{th:taudPV}, \ref{th:taudPVI} remain valid in various
degenerate situations. For example, zero at $z=a_1$ and pole at
$z=b_1$ can coalesce, giving a singularity of the type considered in
Section \ref{sc:zeroandpole}. This degeneration is used in Section
\ref{sc:Hahn}.
\end{remark*}

\section{Gap probabilities}\label{sc:probability}

The goal of this section is to show that tau-functions naturally
arise as the gap probabilities in the discrete probabilistic models
of random matrix type.

\subsection{}
Fix a finite set $\cX\subset\C$ (the phase space), and two families
of weight functions
$\omega_{1,1},\dots,\omega_{1,p},\omega_{2,1},\dots,\omega_{2,q}$
defined on $\cX$. Assume the weight functions have no zeroes on
$\cX$. Also, fix two multi-indices ${\mathbf
n}=(n_1,\dots,n_p),{\mathbf m}=(m_1,\dots,m_q)$ such that
$$N=\sum_{i=1}^p n_i=\sum_{i=1}^qm_i.$$

Set
$$
F(x_1,\dots,x_N)=\det[\phi_i(x_j)]_{i,j=1}^N\det[\psi_i(x_j)]_{i,j=1}^N,
$$
where \begin{align*} \{\phi_i(x)\mid
i=1,\dots,N\}&=\{\omega_{1,i}(x)x^j\mid
i=1,\dots,p,j=0,\dots,n_i-1\}, \\
\{\psi_i(x)\mid i=1,\dots,N\}&=\{\omega_{2,i}(x)x^j\mid
i=1,\dots,q,j=0,\dots,m_i-1\}.
\end{align*}

We always make the following basic assumption:
\begin{equation}
Z=\sum_{x_1,\dots,x_N\in\cX}F(x_1,\dots,x_N)\ne 0.
\label{eq:basicassumption}
\end{equation}
\begin{remark*} Let $\cF$ (resp. $\cG$) be the subspace of $\ell^2(\cX)$ spanned by
$\phi_i$'s (resp. $\psi_i$'s). Then \eqref{eq:basicassumption} is
equivalent to $\dim(\cF)=\dim(\cG)=N$ and $\cF\cap\cG^\perp=\{0\}$.
\end{remark*}

\begin{lemma} Let $K(x,y)$ be the matrix of the
projection in $\ell^2(\cX)$ onto $\cF$ parallel to $\cG^\perp$:
$$
K(x,y)=\sum_{i,j=1}^NM_{ij}\phi_i(x)\psi_j(y)\quad\text{for}\quad
M=\Vert\langle\phi_i,\psi_j\rangle\Vert^{-t}_{i,j=1,\dots,N}.$$ Then
for any subset $\cY\subset\cX$
$$
\frac{1}{Z}\sum_{x_1,\dots,x_N\in\cY}F(x_1,\dots,x_N)=
\det\left((1-K)\bigl|_{\ell^2(\cX-\cY)}\right).
$$
\label{lm:gapprobability}
\end{lemma}
The proof is a standard argument in the random matrix theory.

\subsection{}
Consider on $\p1$ the vector bundle
$$\vL_\varnothing=\vO(n_1-1)\oplus\dots\oplus\vO(n_p-1)\oplus\vO(-m_1-1)\oplus\dots
\oplus\vO(-m_q-1).$$ For any subset $\cY\subset\cX$, define a
modification $\vL_\cY$ of $\vL_\varnothing$ by
\begin{enumerate}
\item $\vL_\cY$ and $\vL_\varnothing$ coincide on $\p1\setminus\cY$;

\item Near any $y\in\cY$, sections of $\vL_\cY$ are rational sections
$$s=(s_{1,1},\dots,s_{1,p};s_{2,1},\dots,s_{2,q})^t\in\vL_\varnothing$$
such that $s_{1,i}$ is regular at $y$ ($i=1,\dots,p$), $s_{2,i}$ has
at most a first order pole at $y$ ($i=1,\dots,q$), and
$$\res_y(s_{2,i})=\omega_{2,i}(y)\cdot\sum_{j=1}^p\omega_{1,j}(y)s_{1,j}(y).$$
\end{enumerate}
Note that $\deg(\vL_\cY)=\deg(\vL_\varnothing)=-p-q$.

\begin{proposition}
Under the assumption \eqref{eq:basicassumption},
$\vL_\cX\simeq(\vO(-1))^{p+q}$. \label{pp:basicassumption}
\end{proposition}
\begin{proof}
This follows from a discrete version of \cite[Theorem 3.1]{DK}.
Since we do not need an explicit solution to the associated
Riemann-Hilbert problem, we provide an independent argument.

Since $\deg(\vL_\cX)=-p-q$, it suffices to show that $\vL_\cX$ has
no global sections. A global section of $\vL_\cX$ is of the form
$$
s=(s_{1,1},\dots,s_{1,p};s_{2,1},\dots,s_{2,q})^t,
$$
where $s_{1,i}$ is a polynomial in $z$ of degree at most $n_i-1$
($i=1,\dots,p$), and $s_{2,i}$ is given by
$$
s_{2,i}(z)=\sum_{x\in\cX}\frac{\omega_{2,i}(x)\cdot\sum_{j=1}^p\omega_{1,j}(x)s_{1,j}(x)}
{z-x},\qquad (i=1,\dots,q).
$$
and satisfies the following condition:
\begin{equation}
\text{The order of zero of $s_{2,i}(z)$ at $z=\infty$ is at least
$m_i+1$.} \label{cn:zero}
\end{equation}

Equivalently, \eqref{cn:zero} means that for any polynomial $p(z)$
of degree $m_i-1$ or less,
$$\res\limits_{z=\infty}s_{2,i}(z)p(z)=0.$$
Evaluating the residue as the sum over finite poles, we obtain
$$\sum_{x\in\cX}p(x)\omega_{2,i}(x)\cdot\sum_{j=1}^p\omega_{1,j}(x)s_{1,j}(x)=0,\qquad(i=1,\dots,p;\quad\deg(p)\le m_i-1).$$
Equivalently, $\sum_{j=1}^p\omega_{1,j}(x)s_{1,j}(x)$ belongs to
$\cG^\perp\cap\cF$, which is trivial by our assumption.
\end{proof}

Similarly to Section \ref{sc:zeroandpole}, we introduce at every
point $x\in\cX$ a vector $w_x$ and a functional $w_x'+w_x''(z-x)$:
\begin{equation}
\begin{aligned}
w_x&=(0,\dots,0;\omega_{2,1}(x),\dots,\omega_{2,q}(x))^t\\
w_x'&=(\omega_{1,1}(x),\dots,\omega_{1,p}(x);0,\dots,0)^t\\
w_x''&=\left(0,\dots,0;\tfrac{1}{\omega_{2,1}(x)}\,,0,\dots,0\right)^t.
\end{aligned}
\label{eq:w}
\end{equation}

\begin{remark*} In what follows, $w_x''$ is important only modulo
$w_x^\perp$. In this sense, the definition of $w_x''$ is symmetric:
$$\left(0,\dots,0;\tfrac{1}{\omega_{2,1}(x)}\,,0,\dots,0\right)^t\equiv\dots\equiv
\left(0,\dots,0;0,\dots,0,\tfrac{1}{\omega_{2,q}(x)}\right)^t \mod
w_x^\perp.$$
\end{remark*}

The modification $\vL_\cY$ (for $\cY\subset\cX$) can be described in
terms of these data as follows: the sections of $\vL_\cY$ near
$y\in\cY$ are sections $s\in\vL_\varnothing$ with at most a first
order pole such that $\res_y s\in\C w_y$ and
$(w'_y+w''_y(z-y))s|_{z=y}=0$. Note that $\langle
w_y,w'_y\rangle=0$, so the last condition makes sense.

\begin{theorem}
For any $\cY\subset\cX$, we have
$$\frac{\tau(\vL_\cY)}{\tau(\vL_\cX)}=
\det\left(1-K|_{\ell^2(\cX-\cY)}\right).$$ Here we identify
$\detrg(\vL_\cY)=\detrg(\vL_\cX)$ by means of $(w_x,w'_x,w''_x)$,
$x\in\cX\setminus\cY$, so we can view the left-hand side as a
number, see below. \label{th:det=tau}
\end{theorem}

Let us describe the identification $\detrg(\vL_\cY)=\detrg(\vL_\cX)$
explicitly. Let $\vL_\cY^{up}$ be a modification of $\vL_\cY$ on
$\cX-\cY$ whose sections near $x\in\cX-\cY$ are of the form
$$s=(s_{1,1},\dots,s_{1,p};s_{2,1},\dots,s_{2,q})^t\in\vL_\cY,$$
where $s_{1,i}$ is regular at $x$ ($i=1,\dots,p$), $s_{2,i}$ has at
most a first order pole at $x$ ($i=1,\dots,q$), and
$$\res_xs\sim w_x=(0,\dots,0;\omega_{2,1}(x),\dots,\omega_{2,q}(x))^t.$$
Note that $\vL_\cY^{up}$ is also an upper modification of $\vL_\cX$.

For every point $x\in\cX-\cY$, consider two functionals on sections
of $\vL_\cY^{up}$:
\begin{align*}
 f_x(s)&=(w'_x+w''_x(z-x))s|_{z=x},\\
 g_x(s)&=\res_xs/w_x.
\end{align*}
Note that sections of $\vL_\cX$ (resp. $\vL_\cY$) are exactly
sections of $\vL_\cY^{up}$ on which $f_x$ (resp. $g_x$) vanish for
all $x\in\cX-\cY$. In this way, we get identifications
\begin{align*}
f_{\cX-\cY}&=(f_x)_{x\in\cX-\cY}:H^0(\p1,\vL_\cY^{up}/\vL_\cX)\iso\C^{|\cX|-|\cY|}\\
g_{\cX-\cY}&=(g_x)_{x\in\cX-\cY}:H^0(\p1,\vL_\cY^{up}/\vL_\cY)\iso\C^{|\cX|-|\cY|}.
\end{align*}
This induces an isomorphism (see Example \ref{ex:ratio}):
\begin{equation*}
\detrg(\vL_\cX)\otimes\detrg(\vL_\cY)^{-1}=
\det(H^0(\p1,\vL_\cY^{up}/\vL_\cY))\otimes
\det(H^0(\p1,\vL_\cY^{up}/\vL_\cX))^{-1}=\C.
\end{equation*}
In other words, the ratio ${\tau(\vL_\cY)}/{\tau(\vL_\cX)}$ is the
determinant of the composition
\begin{equation*}
\C^{|\cX|-|\cY|}\iso H^0(\p1,\vL_\cY^{up}/\vL_\cX)\simeq
H^0(\p1,\vL_\cY^{up})\to H^0(\p1,\vL_\cY^{up}/\vL_\cY)\iso
\C^{|\cX|-|\cY|}.
\end{equation*}

\begin{proof}[Proof of Theorem \ref{th:det=tau}]
Define an embedding
$\iota:H^0(\p1,\vL_\cY^{up})\hookrightarrow\ell^2(\cX)$ as follows:
given
$$s=(s_{1,1},\dots,s_{1,p};s_{2,1},\dots,s_{2,q})\in
H^0(\p1,\vL_\cY^{up}),$$ we set ($x\in\cX$)
\begin{gather*}
\phi(x)=\sum_{i=1}^ps_{1,i}(x)\omega_{1,i}(x),\qquad
\psi(x)=\res_x(s)/w_x,\qquad \iota(s)=\phi+\psi.
\end{gather*}
Note that $\phi\in\cF$, $\psi\in\cG^\perp$ (see proof of Proposition
\ref{pp:basicassumption}), so $s$ is uniquely determined by
$\iota(s)$.

The image $\iota(H^0(\p1,\vL_\cY^{up}))$ is the space of functions
supported by $\cX-\cY$. The functionals $f_x$ and $g_x$ can be
written as $f_x(s)=(\iota(s))(x)$, $g_x(s)=(\pi\circ\iota(s))(x)$,
where $\pi:\ell^2(\cX)\to\ell^2(\cX)$ is the projection onto
$\cG^\perp$ parallel to $\cF$ (that is to say,
$\pi(\iota(s))=\psi(x)$). Thus, the ratio
$\tau(\vL_\cY)/\tau(\vL_\cX)$ equals the determinant of the
composition
$$\ell^2(\cX-\cY)\hookrightarrow\ell^2(\cX)
\xrightarrow{\pi}\ell^2(\cX)\to\ell^2(\cX-\cY).$$
\end{proof}

\begin{remark}
\label{rm:Hirota2}
\newcommand\vLud{\vL_\cX^{\uparrow x,\downarrow y}}
The entries of the matrix $1-K$ can be interpreted as ratios of
suitable $\tau$-functions. Namely, fix $x,y\in\cX$, and consider the
modification $\vLud$ of $\vL_\cX$ at $x$ and $y$ such that for $x\ne
y$
\begin{itemize}
\item Sections of $\vLud$ near $x$ are of the form
$$s=(s_{1,1},\dots,s_{1,p};s_{2,1},\dots,s_{2,q})^t\in\vL_\cX,$$
where $s_{1,i}$ is regular at $x$ ($i=1,\dots,p$), $s_{2,i}$ has at
most a first order pole at $x$ ($i=1,\dots,q$), and
$$\res_xs\sim w_x=(0,\dots,0;\omega_{2,1}(x),\dots,\omega_{2,q}(x))^t;$$

\item Sections of $\vLud$ near $y$ are of the form
$$s=(s_{1,1},\dots,s_{1,p};s_{2,1},\dots,s_{2,q})^t\in\vL_\cX,$$
where $s_{1,i}$ is regular at $y$ ($i=1,\dots,p$), $s_{2,i}$ is
regular at $y$ ($i=1,\dots,q$), and
$$\sum_{i=1}^p\omega_{1,i}(y)s_{1,i}(y)=0.$$
\end{itemize}
Thus if $x\ne y$, $\vLud$ is an elementary upper modification of
$\vL$ at $x$ and its elementary lower modification at $y$. If $x=y$,
we set $\vLud=\vL_{\cX-\{x\}}$.

Taking $\cY=\cX-\{x\}$, we see that $\vL_\cY^{up}$ is an upper
modification of both $\vL_\cX$ and $\vLud$. Under $\iota$,
$H^0(\p1,\vL_\cY^{up})$ goes to the (one-dimensional) space of
functions supported by $\{x\}$. The $(x,y)$-entry of $\pi=1-K$ is
the ratio of functionals $g_y/f_x$ on this one-dimensional space,
which equals $\tau(\vLud)/\tau(\vL_\cX)$. This statement can be
viewed as a discrete analog of \cite[Theorem 4.3]{DK}.

From the point of view of Section \ref{sc:Hirota}, Theorem
\ref{th:det=tau} is a Hirota type determinantal identity similar to
Proposition \ref{pp:Hirota}.
\end{remark}

\subsection{}

Suppose now that there are rational functions
$$\varpi_{1,1}(z),\dots,\varpi_{1,p}(z);
\varpi_{2,1}(z),\dots,\varpi_{2,q}(z)$$ such that for any $x\in\cX$
such that $x+1\in\cX$, we have
$$\frac{\omega_{1,i}(x+1)}{\omega_{1,i}(x)}=\varpi_{1,i}(x),\qquad
\frac{\omega_{2,i}(x+1)}{\omega_{2,i}(x)}=\varpi_{2,i}(x).$$ In
particular, $\varpi_{1,i}$ and $\varpi_{2,i}$ are regular nonzero at
$x$.

On
$$\vL_\varnothing=\vO(n_1-1)\oplus\dots\oplus\vO(n_p-1)\oplus\vO(-m_1-1)\oplus\dots
\oplus\vO(-m_q-1),$$ consider the d-connection
$$\vA(z)=\diag\left(\frac{1}{\varpi_{1,1}(z)},\dots,\frac{1}{\varpi_{1,p}(z)},
\varpi_{2,1}(z),\dots,\varpi_{2,q}(z)\right).$$

For $a,b\in\C$ such that $b-a\in\Z_{>0}$, we call the set
${[a,b]}_\Z=\{a,a+1,\dots,b\}$ the \emph{integral segment} with
endpoints $a$ and $b$. Suppose ${\mathbf a}=(a_1,\dots,a_n),{\mathbf
b}=(b_1,\dots,b_n)$ are such that ${[a_i,b_i]}_\Z$, $i=1,\dots,n$,
are non-intersecting integral segments contained in $\cX$. Set
$$
D({\mathbf a},{\mathbf
b})=\frac{1}{Z}\sum_{x_1,\dots,x_N\in\cX-\bigcup_i{[a_i,b_i]}_\Z}F(x_1,\dots,x_N)=
\det\left(1-K\bigl|_{\ell^2(\bigcup_i[a_i,b_i]_\Z)}\right)
$$
(see Lemma \ref{lm:gapprobability} for the last equality).

Set $\cY=\cX-\bigcup_i{[a_i,b_i]}_\Z$, and consider $\vA(z)$ as a
d-connection on $\vL_\cY$. For every $i$ such that $b_i+1\in\cY$,
the connection $\vA(z)$ has at $z=b_i$ a singularity of type
described in Section \ref{sc:zeroandpole}. The corresponding
modification of $\vL_\cY$ is exactly $\vL_{\cY-\{b_i+1\}}$.
Following
Section \ref{sc:zeroandpole}, we need to choose at $z=b_i+1$ a
vector $w$ and a functional $w'+(z-b_i-1)w''$. It is natural to set
$w=w_{b_i+1}$, $w'=w'_{b_i+1}$, $w''=w''_{b_i+1}$, with the
right-hand sides defined by \eqref{eq:w}. Moreover, it is explained
in Section \ref{sc:zeroandpole} that a choice of $(w,w',w''\mod
w^\perp)$ at $z=b_i+1$ determines the corresponding choice at
$z=b_i+2$. Our definition of $\vA(z)$ is such that the new triple is
exactly $w_{b_i+2}$, $w'_{b_i+2}$, $w''_{b_i+2}$. Here we assume
that $b_i+2\in\cY$.

Similar arguments apply to $a_i$. These observations imply the
following statement.

\begin{theorem}
Suppose that\begin{equation}
\vL_\cY\simeq(\vO(-1))^{p+q},\label{eq:tr}\end{equation} and let
$A(z)$ be the matrix of $\vA(z)$ corresponding to a choice of
isomorphism \eqref{eq:tr}. Then $D({\mathbf a,\mathbf b})$ is a
$\tau$-function of $A(z)$ in the sense that its second difference
logarithmic derivatives
$$\frac{D(\mathbf a,\mathbf b+e_i+e_j)\cdot D(\mathbf a,\mathbf b)}
{D(\mathbf a,\mathbf b+e_i)\cdot D(\mathbf a,\mathbf b+e_j)}$$ can
be computed using the recipe of Section \ref{sc:zeroandpole}. The
same statement holds true for derivatives with respect to $\mathbf
a$ and for the mixed derivatives. \qed
\end{theorem}

\section{Example: Hahn orthogonal polynomial ensemble}
\label{sc:Hahn} In the notation of the previous section, take
$\cX=\{0,\dots,M\}$, $p=q=1$, $m_1=n_1=N$. Set
$$\omega_1(x)=\frac{\Gamma(\alpha+x+1)}{x!},\qquad
\omega_2(x)=\frac{\Gamma(\beta+M-x+1)}{(M-x)!}.$$ This corresponds
to $$F(x_1,\dots,x_N)=\prod_{1\le i<j\le
N}(x_i-x_j)^2\prod_{i=1}^N\omega_1(x_i)\omega_2(x_i).$$ Note that if
$\alpha$ and $\beta$ are such that $\omega_1(x)\omega_2(x)>0$ for
$x\in\cX$, $\omega_1\omega_2$ is the weight function for the
classical Hahn orthogonal polynomials.

Set $$D(s)=\frac{1}{Z}\sum_{x_1,\dots,x_N\le s}F(x_1,\dots,x_N).$$

\begin{theorem} For generic $\alpha$ and $\beta$, there exist
sequences $(q_s,r_s)$, where $s=N-1,\dots,M$, that satisfy the
difference $P_{VI}$ of Proposition \ref{pp:dPVI} with
\begin{gather*}a_1=s,\quad a_2=-1,\quad a_3=M,\\ b_1=s,\quad
b_2=-\alpha-1,\quad b_3=\beta+M,\\ d_1+b_1+b_2+b_3=-\alpha-N,\quad
d_2+b_1+b_2+b_3=\beta+N,
\end{gather*}
such that the second derivative of $D(s)$ is given by Theorem
\ref{th:taudPVI}. Here \begin{gather*} q=q_s,\quad q'=q_{s-1},\qquad
r=r_s,\quad r'=r_{s-1},\\ \tau= D(s),\quad \tau'=D(s-1),\quad
\tau''=D(s-2).
\end{gather*}
\label{th:Hahn}
\end{theorem}

\begin{remarks*}
1. We assume that $\alpha$ and $\beta$ are generic so that all
bundles involved are isomorphic to $(\vO(-1))^2$. However, one can
view $\alpha$ and $\beta$ as parameters and the statement of Theorem
\ref{th:Hahn} as an identity between rational functions. If $\alpha$
and $\beta$ are such that $\omega_1(x)\omega_2(x)>0$ on $\cX$, all
bundles are isomorphic to $(\vO(-1))^2$ by Proposition
\ref{pp:basicassumption}.

2. The initial conditions for the recurrences (that is, $p_{N-1}$,
$q_{N-1}$, $D(N-1)$, and $D(N)$) can be explicitly evaluated using
the algorithm of \cite[Section 6]{BB}.

3. Consider the limit $M\to\infty$. If we scale the lattice $\cX$ by
$M^{-1}$, the Hahn orthogonal polynomials converge to the Jacobi
orthogonal polynomials on $[0,1]$ (with same parameters $\alpha$,
$\beta$), and $D(s)$ converges to the corresponding quantity for the
Jacobi polynomial ensemble. At the same time, the d-connections
become ordinary connections and discrete isomonodromy
transformations converge to the continuous isomonodromy
deformations, as in Section \ref{sc:limit}; see also \cite[Section
5]{B2}. In the one-interval case $\cY=\{s+1,\dots,M\}$, this
corresponds to the degeneration of $dPVI$ (from Theorem
\ref{th:Hahn}) into classical $PVI$. This degeneration is described
in \cite[Section 6.4]{AB}. In the continuous setting, a description
of the relation between isomonodromy transformation and the Jacobi
polynomial ensemble, including the $PVI$ case, can be found in
\cite[Section 8.1]{BD}.
\end{remarks*}

\begin{proof}
We set
$$\varpi_1(z)=\frac{\alpha+z+1}{z+1},\qquad\varpi_2(z)=\frac{z-M}{z-M-\beta}.$$
Thus the matrix $A(z)=\diag(\varpi_1(z)^{-1},\varpi_2(z))$ has
simple zeroes at $-1$ and $-M$, simple poles at $-\alpha-1$,
$\beta+M$, and it behaves at infinity as
$1+\diag(-\alpha,\beta)/z+O(z^{-2})$. If we now consider the
corresponding $d$-connection $\vA$ on $\vL_\cY$ for
$\cY=\{s+1,\dots,M\}$, we see that it has simple zeroes at $-1$,
$-M$, simple poles at $-\alpha-1$, $\beta+M$, and that at $z=s$, its
singularity is of the type considered in Section
\ref{sc:zeroandpole}. Finally, on the formal neighborhood of
infinity, there exists a trivialization $\vR(z):\C^2\to\vL_z$ such
that the matrix of $\vA$ with respect to $\vR$ equals
$$\vR(z+1)^{-1}\vA(z)\vR(z)=\begin{bmatrix}1+\frac{-\alpha-N+1}{z}&0\\0&1+\frac{\beta+N+1}{z}\end{bmatrix}.
$$ Proposition \ref{pp:dPVI}, Theorem \ref{th:taudPVI}, and Theorem
\ref{th:det=tau} conclude the proof.
\end{proof}

\newcommand{\eprint}[1]{{\tt
#1}}

\end{document}